\documentclass[12pt,a4paper,oneside]{amsart}
\usepackage[utf8]{inputenc}
\usepackage[english]{babel}
\usepackage[T1]{fontenc}
\usepackage{amsmath, dsfont, amsthm}
\usepackage{mathtools}
\usepackage{amsfonts}
\usepackage{amssymb}
\usepackage{geometry}
\usepackage{fullpage}
\usepackage{mathrsfs, enumerate, csquotes, color}
\usepackage{graphicx,color,subfigure}
\usepackage{xcolor}
\definecolor{vertfonce}{rgb}{0.20, 0.46, 0.25}
\definecolor{rougefonce}{rgb}{0.64, 0.09, 0.20}
\definecolor{gris}{gray}{0.5}
\usepackage[breaklinks=true,
colorlinks=true,
linkcolor=rougefonce,
citecolor=vertfonce]{hyperref}

\usepackage{pgfplots}
\pgfplotsset{compat=1.15}
\usepackage{mathrsfs}
\usetikzlibrary{arrows}

\newcommand{\Bl}{\color{blue}}
\newcommand{\Bk}{\color{black}}

\renewcommand{\Re}{\mathrm{Re}\,}
\renewcommand{\Im}{\mathrm{Im}\,}
\renewcommand{\leq}{\leqslant}	
\renewcommand{\geq}{\geqslant}

\newtheorem{theorem}{Theorem}[section]
\newtheorem{lemma}[theorem]{Lemma}
\newtheorem{corollary}[theorem]{Corollary}
\newtheorem{proposition}[theorem]{Proposition}

\theoremstyle{definition}
\newtheorem{remark}[theorem]{Remark}

\newtheorem{definition}[theorem]{Definition}
\newtheorem{assumption}{Assumption}

\makeatletter

\@addtoreset{equation}{section}
\makeatother

\newcommand{\N}{{\mathbb N}}

\newcommand{\R}{{\mathbb R}}

\usepackage[normalem]{ulem}
\definecolor{DarkGreen}{rgb}{0,0.5,0.1} % David

\newcommand\soutD{\bgroup\markoverwith
{\textcolor{DarkGreen}{\rule[.5ex]{2pt}{1pt}}}\ULon}
\newcommand{\Hm}[1]{\leavevmode{\marginpar{\tiny%
$\hbox to 0mm{\hspace*{-0.5mm}$\leftarrow$\hss}%
\vcenter{\vrule depth 0.1mm height 0.1mm width \the\marginparwidth}%
\hbox to
0mm{\hss$\rightarrow$\hspace*{-0.5mm}}$\\\relax\raggedright #1}}}

\title[Semiclassical asymptotics of the Bloch--Torrey operator]{Semiclassical asymptotics of the Bloch--Torrey operator in two dimensions}

\author[F. H\'erau]{Fr\'ed\'eric H\'erau}
\address[F. H\'erau]{LMJL - UMR6629, Nantes University  , CNRS, 2 rue de la Houssini\`ere, BP 92208, F-44322 Nantes cedex 3, France}
\email{herau@univ-nantes.fr}

\author[D. Krej\v{c}i\v{r}{\'\i}k]{David Krej\v{c}i\v{r}{\'\i}k}
\address[D. Krej\v ci\v r\'ik]{Department of Mathematics,
	Faculty of Nuclear Sciences and Physical Engineering,
	Czech Technical University in Prague,
	Trojanova 13, 12000 Prague, Czech Republic}
\email{david.krejcirik@fjfi.cvut.cz}

\author[N. Raymond]{Nicolas Raymond}
\address[N. Raymond]{Univ Angers, CNRS, LAREMA, Institut Universitaire de France, SFR MATHSTIC, F-49000 Angers, France}
\email{nicolas.raymond@univ-angers.fr}

\begin{document}

\begin{abstract}
The Bloch--Torrey operator $-h^2\Delta+e^{i\alpha}x_1$
on a bounded smooth planar domain,
subject to Dirichlet boundary conditions, is analyzed.
Assuming $\alpha\in\left[0,\frac{3\pi}{5}\right)$
and a non-degeneracy assumption on
the left-hand side of the domain,
asymptotics of the eigenvalues with the smallest real part
in the limit $h \to 0$ are derived.
The strategy is a backward complex scaling
and the reduction to a tensorized operator
involving a real Airy operator and a complex harmonic oscillator.
\end{abstract}	
	
\maketitle
	
%----------------------%	
\section{Introduction}	
%----------------------%
%
Let~$\Omega$ be a smooth bounded open connected set in $\mathbb{R}^2$.
Given a small positive parameter~$h$
and a fixed real constant $\alpha \in [0,\pi]$,
we consider the operator
\begin{equation}\label{operator}
  \mathscr{L}_{h,\alpha}=-h^2\Delta+e^{i\alpha}x_1
\end{equation}
in  $L^2(\Omega)$,
subject to Dirichlet boundary conditions.
On its natural domain
$
  \mathrm{Dom}(\mathscr{L}_{h,\alpha})=H^2(\Omega)\cap H^1_0(\Omega)
$,
the operator is closed, has non-empty resolvent set and compact resolvent.
Consequently, the spectrum is purely discrete
and can be written as an infinite sequence of complex numbers tending
to $+\infty$ in modulus or as a (possibly empty) finite sequence.
The latter cannot be \emph{a priori} excluded because $\mathscr{L}_{h,\alpha}$
is non-selfadjoint unless $\alpha \in \{0,\pi\}$.
Our goal is to show the existence of ``low-lying'' eigenvalues
and derive their asymptotics in the semiclassical limit $h \to 0$.

\subsection{Motivations}
There are two sources of motivation for this work.
First, the selfadjoint situation $\alpha=0$ has been recently analysed
in~\cite{CKPRS22}
in the context of semiconductor devices exposed
to a strong uniform electric field.
Indeed, $h^{-2} \mathscr{L}_{h,0}$ is the Hamiltonian
of an electron confined to a nanostructure of shape~$\Omega$,
subject to singularly scaled electric potential $h^{-2} x_1$.
The following geometric hypothesis is adopted in~\cite{CKPRS22}:
	\begin{assumption}\label{geoassumption0}
	The minimum	$\min\{ x_1 : x\in\overline{\Omega} \}$
	is uniquely attained at a point $A_0$, assumed to be $(0,0)$ (without loss of generality). Moreover, the (signed) curvature $\kappa_0$ of $\partial\Omega$
(computed with respect to the inner normal of~$\Omega$)
	at $A_0 = (0,0)$ is positive.
	\end{assumption}

Let $(\lambda_n(h))_{n\geq 1}$ denote the non-decreasing sequence of
the eigenvalues of~$\mathscr{L}_{h,0}$,
where each eigenvalue is repeated according to its multiplicity.
The following asymptotic estimate of each individual eigenvalue
was established in~\cite{CKPRS22}:
\begin{theorem}[\cite{CKPRS22}]\label{thm.0}
	 Assume $\alpha = 0$ and Assumption~\ref{geoassumption0}.
	Then, for all $n\geq 1$,
\begin{equation}\label{as1}
		\lambda_n(h)=z_1 h^{\frac23}+(2n-1)h\sqrt{\frac{\kappa_0}{2}}+o(h)
\end{equation}
	as $h \to 0$,
	where $z_1$ is the absolute value of
	the smallest zero of the Airy function $\mathsf{Ai}$.	
\end{theorem}	

The eigenvalue splitting given by the second term
containing the curvature is experimentally spectacular,
for it enables one to determine the shape of a convex nanostructure
by imposing uniform electric fields in various directions~\cite{CKPRS22b}.

Second, there have been an intensive study of the operator~\eqref{operator}
for the purely imaginary choice $\alpha=\frac{\pi}{2}$
in various geometric settings
(and even for more general electric potentials)
\cite{A08,
Almog-Helffer-Pan_2013,Almog-Helffer-Pan_2012,
H13, AH16, GH18,
AGH19,Almog-Grebenkov-Helffer_2018,
Almog-Helffer_2020,Grebenkov-Moutal-Helffer,
Semoradova-Siegl}.
Among the variety of physical motivations mentioned in these references,
let us point out the Bloch--Torrey equation describing
the diffusion-precession of spin-bearing
particles in nuclear magnetic resonance experiments.

In particular,
in \cite[Theorem 1.1]{GH18}, quasimodes are constructed and allow to conjecture the behavior of the eigenvalues with the smallest real part. Motivated by these constructions, the behavior of the real part of the left-most spectrum has then been analyzed, see \cite[Theorem 1.6]{AGH19}. For our linear electric potential, we can apply, for instance, \cite[Theorem 4.1.1]{H13} and \cite[Theorem 1.1]{AH16}, and we get the following typical result.
	\begin{theorem}[\cite{AH16}]\label{thm.AGH}
Assume $\alpha=\frac{\pi}{2}$ and Assumption \ref{geoassumption0}.
Then
\begin{equation}\label{as2}
		\inf\Re \mathrm{sp}(\mathscr{L}_ {h,\alpha})
		= \frac{z_1h^{\frac23}}{2} + o(h^{\frac23})
\end{equation}
as $h \to 0$.		
	\end{theorem}
As observed in \cite[Introduction]{AH16},
the lower bound in Theorem \ref{thm.AGH} can be proved without Assumption \ref{geoassumption0} (see also \cite{A08, H13}).

\bigskip

In this article we explain
the transition between $\alpha=0$ (Theorem \ref{thm.0})
and $\alpha=\frac{\pi}{2}$ (Theorem \ref{thm.AGH}).
First of all, we show how~$\alpha$
enters the constant coefficient in the first term of
the asymptotic expansions~\eqref{as1} and~\eqref{as2}.
We also aim at providing the reader with an accurate description of the spectrum by exhibiting spectral gaps in the left-most part of the spectrum
(similarly to Theorem \ref{thm.0} in the case when $\alpha=0$).
This question is all the more interesting that, when $\alpha \in (0,\pi)$, the operator $\mathscr{L}_{h,\alpha}$ is not selfadjoint and therefore, classical tools and strategies such as the min-max and spectral theorems
(used, for instance, in~\cite{CKPRS22}) have to be replaced by unconventional arguments.
Throughout this paper, we use the nickname Bloch--Torrey operator
for~\eqref{operator} even if $\alpha\not=\frac{\pi}{2}$.

\subsection{Heuristics}\label{sec.heur}
	Before stating our main results, let us explain the intuitive origin of Theorems \ref{thm.0} and \ref{thm.AGH}. This is also the opportunity to discuss the heuristics of our main theorem, which is stated in Section \ref{sec.mainresult}.
	
	 When $\alpha=0$ and under Assumptions \ref{geoassumption0}, due to the Agmon estimates, we can check that the eigenfunctions associated with the lowest eigenvalues are localized near $A_0$.  We will see that such a localization behavior persists in some sense for certain eigenfunctions when $\alpha\in[0,\pi]$, especially for those associated with the left-most eigenvalues when $\alpha\in\left[0,\frac\pi2\right)$. Anyway, this naively suggests to use the classical tubular coordinates near the (outer) boundary
defined through the map
\begin{equation} \label{def.gamma}
\Gamma(s,t)=\gamma(s)-t\mathbf{n}(s)=(\Gamma_1(s,t), \Gamma_2(s,t)),
\end{equation}
where~$\gamma$ is the arc-length parametrization of the outer boundary of $\Omega$, denoted by $\partial\Omega_0$,
and~$\mathbf{n}$ is the outward pointing normal of~$\Omega$. Let $L>0$ be the half-length of $\partial\Omega_0$ and consider the torus $\mathbb{T}_{2L}=\mathbb{R}/(2L\mathbb{Z})$. The map $\Gamma$ induces a smooth diffeomorphism from
$B_{\delta_0}=\mathbb{T}_{2L}\times(0,\delta_0)$
to the tubular neighborhood $T_{\delta_0}$
of width $\delta_0 > 0$ of $\partial\Omega_0$
lying inside~$\Omega$. 
In the coordinates $(s,t)$, the operator~\eqref{operator} becomes
	\[-h^2(1-t\kappa(s))^{-1}\partial_t(1-t\kappa(s))\partial_t-h^2(1-t\kappa(s))^{-1}\partial_s(1-t\kappa(s))^{-1}\partial_s+e^{i\alpha}\Gamma_1(s,t)\,,\]
	acting in the Hilbert space $L^2(B_{\delta_0},(1-t\kappa(s))\mathrm{d}s\mathrm{d}t)$.
Here the curvature function~$\kappa$ is defined via
the Frenet formula $\mathbf{n}'=\kappa \gamma'$.	
	
	 According to Assumption \ref{geoassumption0}
	 (which involves $\Gamma(0,0)=A_0$),
	 we have $\Gamma_1(s,t)=t+\frac{\kappa_0}{2}s^2+\mathscr{O}(ts^2+|s|^3)$. Since $1-t\kappa(s)\simeq 1$ when $t$ is small, this suggests to consider the operator
	\begin{equation}\label{eq.modelnaif}
		\mathscr{P}_{h,\alpha}=-h^2\partial_s^2+ -h^2\partial^2_t+e^{i\alpha}\left(\frac{\kappa_0}{2}s^2+t\right)\,,
	\end{equation}
	acting on $L^2(\mathbb{R}^2_+,\mathrm{d}s\mathrm{d}t)$,
	subject to Dirichlet boundary condition at $t=0$.

Taking profit of the analyticity (since it  is linear) in the variable~$t$,
we make the formal dilation $t=ue^{-i\alpha/3}$.
The model operator $\mathscr{P}_{h,\alpha}$ then becomes
\begin{equation}\label{eq.Nha0}
\mathscr{N}_{h,\alpha}=e^{\frac{2i\alpha}{3}}(h^2D_u^2+u)+h^2D_s^2+e^{i\alpha}\frac{\kappa_0 s^2}{2}\,,
\end{equation}
which is, up to multiplications by complex constants, the sum of a real Airy operator and a complex harmonic oscillator, whose resolvent and spectra are rather well-known. \emph{Heuristically}, this allows us to describe the spectrum of $\mathscr{L}_{h,\alpha}$ accurately in appropriate regions of the complex plane.		
	
\subsection{The main result}\label{sec.mainresult}
The main result of this article is the following theorem,
which can be guessed from the heuristics of the previous section.

	\begin{theorem}\label{thm.1}
	Consider $\alpha\in\left[0,\frac{3\pi}{5}\right)$ and $R>0$ with $R \not\in (2\N-1)\sqrt{\frac{\kappa_0}{2}}$. Under Assumption \ref{geoassumption0}, there exist $h_0>0$ and $N\in\mathbb{N}$ such that for all $h\in(0,h_0)$ the following holds. The spectrum of $\mathscr{L}_{h,\alpha}$ lying in the disk $D(h^{\frac23}e^{2i\alpha/3}z_1, Rh)$ is made of exactly $N$ eigenvalues of algebraic multiplicity $1$ and they satisfy, for all $n\in\{1,\ldots,N\}$,
		\begin{equation}\label{eq.eigenvalueasymptotics}
		\lambda_n(\alpha,h)=h^{\frac23}e^{2i\alpha/3}z_1+(2n-1)he^{i\alpha/2}\sqrt{\frac{\kappa_0}{2}}+o(h)
		\end{equation}
		as $h \to 0$.
		Moreover, for all $\alpha\in[0,\frac{\pi}{2})$, there exist $C, h_0>0$ such that, for all $h\in(0,h_0)$, we have
		\begin{equation}\label{eq.lowerbound}
		\inf\Re \mathrm{sp}(\mathscr{L}_ {h,\alpha})\geq z_1h^{\frac23}\cos\left(\frac{2\alpha}{3}\right)-Ch^{\frac43}\,.
		\end {equation}
		In particular,
				\[\inf\Re \mathrm{sp}(\mathscr{L}_ {h,\alpha})= z_1h^{\frac23}\cos\left(\frac{2\alpha}{3}\right)
				+o(h^{\frac23}) \]
		as $h \to 0$.		
	\end{theorem}

Theorem \ref{thm.1} is illustrated on Figure \ref{fig.1}: there is exactly one eigenvalue (with algebraic multiplicity) in each small circle (which has radius $o(h)$) and there is no spectrum in the gray region when $\alpha\in\left[0,\frac\pi2\right)$.

\begin{remark}~
	\begin{enumerate}[\rm (i)]
\item Theorem \ref{thm.1} gives an accurate description of the spectrum in large balls of size $h$ when $\alpha\in\left[0, \frac{3\pi}{5}\right)$, but it only states the one-term asymptotics of the eigenvalue with the
smallest real part
when $\alpha\in\left[0,\frac{\pi}{2}\right)$. When $\alpha \in \left[0,\frac{\pi}{2}\right)$, we will see that elliptic estimates using the real part of an operator
(which is isospectral to $\mathscr{L}_{h,\alpha}$)
are enough to establish the semiclassical localization near $(0,0)$ (in the Agmon sense) of the eigenfunctions associated with eigenvalues having a real part less than $Mh^{\frac23}$. This localization is the key to get the lower bound \eqref{eq.lowerbound}.  When $\alpha \in \left[\frac\pi2,\frac{3\pi}{5}\right)$, these considerations must be slightly adapted by introducing a parameter $\beta$ and by multiplying the operator by $e^{-i\beta}$. This rotation is the reason why the control of infimum of the real part is lost with our method. This aspect is discussed in more detail in Section \ref{sec.conse}.
\item Our assumptions allow us to deal with the case $\alpha=\frac\pi2$ and to get the asymptotic estimate \eqref{eq.eigenvalueasymptotics}. For more general potentials, see \cite[Theorem 1.1]{AH16}, only the existence of one eigenvalue in the disk is ensured (the one corresponding to $n=1$). Not only our theorem gives the existence of more eigenvalues, it also states that they are algebraically simple and that they are the only ones in the disk. The proof of this simplicity involves rather subtle and tedious elliptic estimates, especially to exclude the existence of Jordan blocks.
\item As we explain in Section \ref{sec.org}, the analysis used to establish Theorem \ref{thm.1} strongly relies on the analyticity of $V(x)=x_1$. However, it seems that arguments such as analytic dilations have not yet been used to investigate the spectrum of such Bloch--Torrey operators on domains. We believe that our method is of independent interest. It can easily be extended to more general analytic potentials~$V$ (still satisfying the generic assumptions in \cite{AH16, AGH19}) and we may even think that it could be used to deal with smooth~$V$ by means of almost analytic extensions.
\item  Unfortunately, our strategy does not allow us to recover Theorem \ref{thm.AGH}, even though, at a formal level, \eqref{eq.lowerbound} would give the appropriate lower bound when $\alpha=\frac\pi2$. In this case, the real parts of two networks of eigenvalues cross, see Section \ref{sec.conse}.
\item Our theorem does not say anything about the eigenfunctions (even if one could prove that they are localized near $A_0$ when $\alpha\in\left[0,\frac\pi2\right)$). Their accurate localization properties (in the Agmon sense) would be quite natural to investigate.
\end{enumerate}
\end{remark}

\begin{figure}[ht!]
	\resizebox{.6\linewidth}{!}{
	\includegraphics{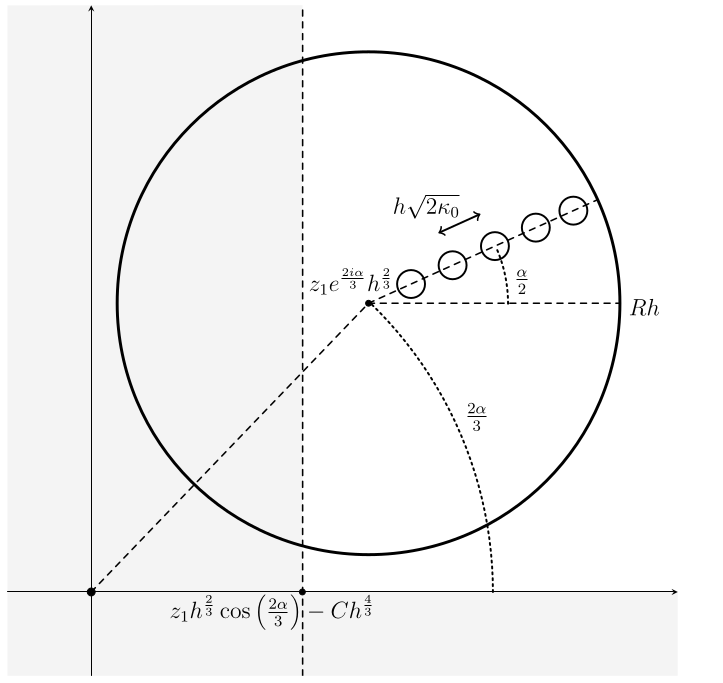}}
	\caption{The spectrum in the disk of center $z_1e^{\frac{2i\alpha}{3}}h^{\frac23}$ and radius $Rh$, when $\alpha \in \left[0, \frac{3\pi}{5}\right)$.}\label{fig.1}
\end{figure}

\subsection{Consequences and extensions}\label{sec.conse}
The analysis in this article can be used to get an \emph{a priori} location of the spectrum in the case when $\alpha \in \left[\frac\pi2,\frac{3\pi}{5}\right)$.

\begin{proposition}\label{prop.lbbeta}
Consider $\alpha\in\left[\frac\pi2,\frac{3\pi}{5}\right)$. There exist $\beta\in\left(0,\frac{\pi}{10}\right]$ with $\frac{2\alpha}{3}\in\left(\beta-\frac\pi2,\beta+\frac\pi2\right)$, $C>0$, and $h\in(0,h_0)$ such that, for all $h\in(0,h_0)$,
\[\inf \Re e^{-i\beta}\mathrm{sp}(\mathscr{L}_{h,\alpha})\geq z_1 h^{\frac23}\cos\left(\frac{2\alpha}{3}-\beta\right)-Ch^{\frac43}\,.\]
In others terms, the eigenvalues $\lambda$ of $\mathscr{L}_{h,\alpha}$ belong to the half-plane given by
\[\cos\beta\,\Re\lambda+\sin\beta\,\Im\lambda\geq z_1 h^{\frac23}\cos\left(\frac{2\alpha}{3}-\beta\right)-Ch^{\frac43}\,.\]
\end{proposition}
Proposition \ref{prop.lbbeta} is illustrated by Figure \ref{fig.2}: there is no spectrum on the left of the dashed oblique line. In fact, in this case, there exist eigenvalues with a smaller real part (as one can see on the same figure). They are related to the right-most part of the domain.

\begin{figure}[ht!]
		\resizebox{.6\linewidth}{!}{
	\includegraphics{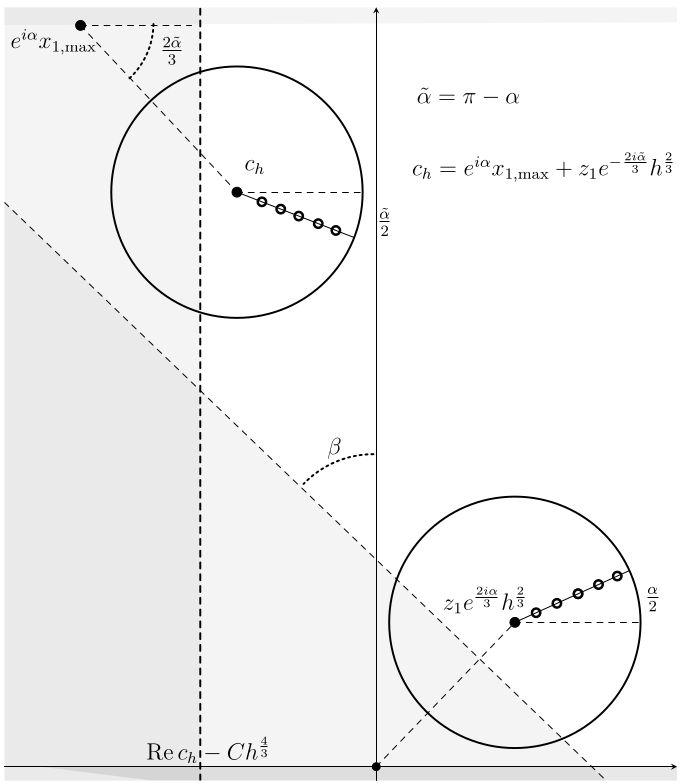}}
	\caption{Case when $\alpha\in\left(\frac\pi2,\frac{3\pi}{5}\right)$.}\label{fig.2}
\end{figure}

\begin{assumption}\label{geoassumption2}
		The maximum	
		$\max\{ x_1 : x\in\overline{\Omega} \} = x_{1,\max}$
		 is uniquely attained at a point~$A_1$. Moreover, the curvature $\kappa_1$ of $\partial\Omega$ at $A_1$ is positive.
	\end{assumption}
	We let $\tilde\alpha=\pi-\alpha$ and we consider $\alpha\in\left(\frac{2\pi}{5},\pi\right]$ so that $\tilde\alpha\in\left[0,\frac{3\pi}{5}\right)$.
		Then, the affine change of variable $y=F(x)= (-x_1 + x_{1,\max},x_2)$ transforms $\mathscr{L}^*_{h,\alpha}$ into the unitarily equivalent operator
	\[U^*\mathscr{L}^*_{h,\alpha}U=\mathscr{L}_{h,\tilde\alpha}+e^{-i\alpha}x_{1,\max}\,,\quad \mathrm{Dom}(\mathscr{L}_{h,\tilde\alpha})=H^2(\tilde\Omega)\cap H^1_0(\tilde\Omega)\,,\quad\tilde\Omega=F(\Omega)\,.\]
Therefore, under Assumption \ref{geoassumption2}, we can apply Theorem \ref{thm.1} to $\mathscr{L}_{h,\tilde\alpha}$ and we get the following.
\begin{corollary}\label{cor.1}
Consider $\alpha\in\left(\frac{2\pi}{5},\pi\right]$ and $R>0$ with $R \not\in (2\N-1)\sqrt{\frac{\kappa_1}{2}}$. Then, the spectrum of $\mathscr{L}_{h,\alpha}$ lying in the disk $D(e^{i\alpha} x_{1,\max} + h^{\frac23}e^{-2i\tilde{\alpha}/3}z_1, Rh)$ is made of $N$ eigenvalues of algebraic multiplicity $1$ and satisfying, for all $n\in\{1,\ldots,N\}$,
\[
\tilde{\lambda}_n(\alpha,h)=e^{i\alpha}x_{1,\max} + h^{\frac23}e^{-2i\tilde{\alpha}/3}z_1
+(2n-1)he^{-i\tilde{\alpha}/2}\sqrt{\frac{\kappa_1}{2}}+o(h)
\]
as $h \to 0$.
Moreover, when $\alpha \in \left(\frac\pi2, \pi\right]$,
\[
\inf\Re \mathrm{sp}(\mathscr{L}_ {h,\alpha})= \cos(\alpha)x_{1,\max} + z_1h^{\frac23}\cos\left(\frac{2\tilde{\alpha}}{3}\right)
 + o(h^{\frac23})
\]
as $h \to 0$.
	\end{corollary}
When $\alpha \in \left(\frac{2\pi}{5}, \frac{3\pi}{5}\right)$, under Assumptions \ref{geoassumption0} and \ref{geoassumption2}, Theorem \ref{thm.1} and Corollary \ref{cor.1} apply. Therefore, we have coexistence of (at least) two networks of eigenvalues. This phenomenon is illustrated on Figure \ref{fig.2} when $\alpha \in \left(\frac{\pi}{2}, \frac{3\pi}{5}\right)$: we see that the right-most part of $\Omega$ determines the left-most part of the spectrum.

\subsection{Organization and strategy}\label{sec.org}
The article is organized as follows.
In Section \ref{sec.2}, we make the analytic dilation argument rigorous.
We introduce an analytic family of operators $(\mathscr{L}_{h,\alpha,\theta})_{\theta\in\Theta}$ in the sense of Kato. To do so, we use a real dilation with respect to the distance to the outer boundary $t$ (acting only near the boundary). This is where we take advantage of the fact that $\Gamma$ (see \eqref{def.gamma}) is always analytic in $t$. When $\theta\in\R$, $\mathscr{L}_{h,\alpha,\theta}$ is isospectral to $\mathscr{M}_{h,\alpha}=\mathscr{L}_{h,\alpha}$ (see Lemma \ref{lem.isoR}). By the Kato theory, it is also isospectral to $\mathscr{L}_{h,\alpha,-i\frac\alpha3}$ (see Corollary \ref{cor.iso}).

In Section \ref{sec.loestimates}, we see that this special choice of complex parameter is particularly convenient since we can prove that the eigenfunctions of $\mathscr{M}_{h,\alpha}$ associated with eigenvalues located in a half-plane of the form $\Re(e^{-i\beta}\lambda)\leq Mh^{\frac23}$ are exponentially localized near $A_0$, see Proposition \ref{prop.loch23}. The introduction of the parameter $\beta$ and the constraint on $\alpha\in\left[0,\frac{3\pi}{5}\right)$ originate from these localization arguments, which are based on ellipticity/coercivity estimates, induced by the complex electric potential (after the change of coordinates and the complex dilation), see Lemma \ref{lem.confinement}.

Section \ref{sec.spectralanalysis} is devoted to the spectral analysis of $\mathscr{M}_{h,\alpha}$. When $\alpha\in\left[0,\frac\pi2\right)$, the asymptotic estimate of the infimum part of the spectrum is obtained, see Proposition \ref{prop.lowerbound} and its proof given in Section \ref{sec.roughloc}. This proves \eqref{eq.lowerbound} (see also Remark \ref{rem.interm}, which proves Proposition~\ref{prop.lbbeta}). When $\alpha\in\left[0,\frac{3\pi}{5}\right)$, we first prove that the spectrum in the disk mentioned in Theorem~\ref{thm.1} is necessarily close (essentially at a distance of order $h^{\frac32}$) to the eigenvalue of a model operator $\mu_n(h,\alpha)=h^{\frac23}e^{2i\alpha/3}z_1+(2n-1)he^{i\alpha/2}\sqrt{\frac{\kappa_0}{2}}$, see Proposition \ref{prop.locspec}. Then, we prove that there is exactly one eigenvalue
(with algebraic muliplicity $1$) in these small discs (see Proposition~\ref{prop.finale} and Figure \ref{fig.1}), and we deduce Theorem \ref{thm.1}.

The proof of Proposition \ref{prop.locspec} relies on three important ingredients. First, it requires resolvent estimates of $\mathscr{N}_{h,\alpha}$ (we recall that $\mathscr{N}_{h,\alpha}$ is given in \eqref{eq.Nha0}), see Proposition \ref{prop.resolventmodel}. The fact that we performed an analytic dilation in $t$ is a crucial help to get the control of the resolvent (by estimating the real part of the operator). The second ingredient is to show that the eigenfunctions associated with the eigenvalues in our disc are good quasimodes for $\mathscr{N}_{h,\alpha}$, see Proposition \ref{prop.quasimode}. To do so, we need to prove optimal localization estimates with respect to the curvilinear abscissa $s$ (see Proposition \ref{prop.locs}) in order to estimate the remainders of order $s^3$ when Taylor expanding the electric potential. We stress that estimating the real part of  $\mathscr{M}_{h,\alpha}$ is a key to get such estimates (and that this argument succeeds thanks to the analytic dilation). Proposition \ref{prop.quasimode} and the resolvent estimate are then enough to locate the spectrum in the small discs.

The fact that the rank of the Riesz projector is at most one requires more work. This is where the third ingredient comes into play. We assume that this rank is at least two and even that we have a Jordan block (in the worst scenario) and we prove that a generalized eigenfunction also satisfies accurate localization estimates, see Section \ref{sec.quasimodes2} and especially Proposition~\ref{prop.quasimode2}. This part of the proof is technically more involved and it is somewhat reminiscent of Caccioppoli estimates, see Proposition \ref{prop.controlf}. There remains to estimate the Riesz projectors to get a contradiction, see Section \ref{sec.Riesz}. To prove that the projectors are non-zero, we consider a quasimode built from the Airy and Hermite functions, see Section \ref{sec.quaRiesz}.

%------------------------------%
\section{The analytic dilation}\label{sec.2}
%------------------------------%	
	
\subsection{The sesquilinear form}	
Before introducing the main idea of this paper,
let us stress that the operator~$\mathscr{L}_{h,\alpha}$ from~\eqref{operator}
is rigorously introduced via its
sesquilinear form defined on $H^1_0(\Omega)$ by
	\[L_{h,\alpha}(\varphi,\psi)=\int_{\Omega}\nabla\varphi\,\overline{\nabla\psi}\,\mathrm{d}x+e^{i\alpha}\int_{\Omega} x_1\varphi\,\overline{\psi}\,\mathrm{d}x\,.\]
Notice that
	\[\mathrm{Re}\,L_{h,\alpha}(\psi,\psi)\geq \|\nabla\psi\|^2
	-\sup_{x\in\Omega}{|x_1|} \, \|\psi\|^2\,,\]
which enables to apply the standard Lax--Milgram theorem.
An elementary argument shows that
	\[\mathrm{sp}(\mathscr{L}_{h,\alpha})\subset\{\lambda\in\mathbb{C} :
	0\leq\Im\lambda\leq (\sin\alpha)\sup_{\Omega}x_1\}\,.\]

	\subsection{An isospectral operator}
Following the intuition described in Section~\ref{sec.heur},
we would like to perform a complex scaling in the normal variable to the outer boundary. By doing that, we will preserve the spectrum as soon as we have a family of type (B) in the sense of Kato \cite[Chap.~VII]{Kato2}.
This will reveal some hidden elliptic properties of the new operator.
	
 Let $\delta\in(0,\delta_0)$ where we recall that $\delta_0$ is defined just after \eqref{def.gamma}. The heuristic considerations of Section \ref{sec.heur} lead to introduce the following unitary transform $\mathscr{U}_\theta$, depending on the real parameter $\theta$. For all $\varphi\in L^2(\Omega)$, we let
\[\mathscr{U}_\theta\varphi=(\varphi_{|\Omega\setminus T_\delta}, \varphi_{|T_\delta}\circ \Gamma(s,J_\theta(u)))\,,\]
with $J_\theta$ given by
\begin{equation*}
t = J_\theta(u)=ue^{\theta\chi(u)}\,,
\end{equation*}
where $\chi$ is non-increasing smooth function
from $[0,\delta]$ to $[0,1]$
such that
 $\chi=1$ near $0$ and $\chi=0$ near $\delta$.
 For all $\epsilon>0$, we can choose $\chi$ so that
\begin{equation}\label{eq.boundchi'}
\|\chi'\|_\infty\leq \frac{1+\epsilon}{\delta}\,.
\end{equation}
Note that $t = J_\theta(u)= ue^{\theta}$ near $0$ and that $t=J_\theta(u)=u$ at a distance larger than $\delta$ of the outer boundary and that the change of variable is smooth in between.
There exists $\theta_0>0$ such that, for all $\theta\in(-\infty,\theta_0)$, the map $J_\theta : (0,\delta)\to(0,\delta)$ is smooth diffeomorphism and, for all $u\in(0,\delta)$,
\begin{equation} \label{eq.jprimetheta}
J'_\theta(u)=(1+\theta u\chi'(u))e^{\theta\chi(u)}>0\,.
\end{equation}
We let
\begin{equation} \label{eq.mtheta}
m_\theta(s,u)=1-J_\theta(u)\kappa(s)\,,
\end{equation}
where $\kappa(s)$  is the curvature at the point of curvilinear coordinate $s$.
Thanks to a change of variables, we have the following.
\begin{lemma}\label{lem.Etheta}
For all $\theta\in(-\infty,\theta_0)$, $\mathscr{U}_\theta$ is an isometry from $L^2(\Omega)$ to the product $E_\theta:=L^2(\Omega\setminus T_\delta)\times L^2(B_\delta,m_\theta(s,u)J'_\theta(u)\mathrm{d}s\mathrm{d}u)$. As vector spaces, we have $E_\theta=E_0=L^2(\Omega\setminus T_\delta)\times L^2(B_\delta)$.
\end{lemma}

Then, let us describe the effect of $\mathscr{U}_\theta$ on the form domain of the operator $\mathscr{L}_{h,\alpha}$, which is $H^1_0(\Omega)$.
\begin{lemma}
We have
\[\begin{split}\mathscr{U}_\theta(H^1_0(\Omega))&=\{(\varphi_1,\varphi_2)\in H^1(\Omega\setminus T_\delta)\times H^1(B_\delta) : \varphi_2(s,0)=0\,\, \&\,\, \varphi_1(\Gamma(s,\delta))=\varphi_2(s,\delta)  \}\\
&=\mathscr{U}_0(H^1_0(\Omega)) \,.\end{split}\]	
\end{lemma}

\begin{proof}
Let us first notice that, by standard trace theorems, the functions $\phi_2$ is well defined a.e. on $\{t=\delta\}$ (where coordinates $t$ and $u$ coincide), as well as $\phi_1 \circ \Gamma$ on $\partial B_\delta$. For a smooth function $\phi$, we have $\varphi_1(\Gamma(s,\delta))=\varphi_2(s,\delta)$ and the result follows by density and the fact that $\mathscr{U}_\theta$ is an isometry.
\end{proof}

Let us now consider the quadratic form induced by $\mathscr{U}_\theta$ from $L_{h,\alpha}$.

\begin{proposition}\label{prop.Uthetaell}
Letting, for all $\varphi\in\mathscr{U}_0(H^1_0(\Omega))$,
\[\begin{split}
	\lefteqn{
	\ell_{h,\alpha,\theta}(\varphi,\varphi)
	=	\int_{\Omega\setminus T_\delta}
	\left( |h\nabla\varphi_1|^2+e^{i\alpha}x_1|\varphi_1|^2 \right)
	\mathrm{d}x
	}
	\\
	&+\int_{B_\delta} (m_\theta^{-2}|h\partial_s\varphi_2|^2+[J'_\theta]^{-2}|h\partial_u\varphi_2|^2+e^{i\alpha}\Gamma_1(s,J_\theta(u))|\varphi_2|^2)m_\theta J'_\theta(u)\mathrm{d}s\mathrm{d}u\,,
	\end{split}\]
we have
\[L_{h,\alpha}(\mathscr{U}^{-1}_\theta\varphi,\mathscr{U}^{-1}_\theta\varphi)=\ell_{h,\alpha,\theta}(\varphi,\varphi)\,.\]
\end{proposition}
\begin{proof}
For all $\varphi\in\mathscr{U}_0(H^1_0(\Omega))$, we let $\psi=\mathscr{U}^{-1}_\theta\varphi\in H^1_0(\Omega)$. Let us first describe the kinetic part:
\[\begin{split}
  \lefteqn{
  \int_{\Omega}|h\nabla\psi|^2\,\mathrm{d}x
=	\int_{\Omega\setminus T_\delta}
  |h\nabla\psi|^2\,\mathrm{d}x+	\int_{T_\delta}|h\nabla\psi|^2\,\mathrm{d}x
  }
\\
&= 	\int_{\Omega\setminus T_\delta}|h\nabla\varphi_1|^2\,\mathrm{d}x+\int_{B_\delta} (m_\theta^{-2}|h\partial_s\varphi_2|^2+[J'_\theta]^{-2}|h\partial_u\varphi_2|^2)m_\theta(s,u)J'_\theta(u)\mathrm{d}s\mathrm{d}u\,.
\end{split}\]	
Using the changes of variable $x \mapsto (s,u)$ on $B_\delta$ for the non-kinetic part completes the proof.
\end{proof}

From Proposition \ref{prop.Uthetaell}, we see that $\mathscr{U}_\theta\mathscr{L}_{h,\alpha}\mathscr{U}^{-1}_\theta$ is the operator associated with $\ell_{h,\alpha,\theta}$ in the ambient space $E_\theta$ (with the weighted scalar product, which depends on $\theta$).
To avoid the $\theta$-dependence of the ambient space through its scalar product, we can consider the isometry
\[\mathscr{V}_\theta: \varphi\mapsto (\varphi_1,\underbrace{m_\theta^{\frac12}(J'_\theta)^{\frac12}\varphi_2}_{=\phi_2})\,,\]
from $L^2(\Omega\setminus T_\delta)\times L^2(B_\delta,m_\theta(s,u)J'_\theta(u)\mathrm{d}s\mathrm{d}u)$ to $L^2(\Omega\setminus T_\delta)\times L^2(B_\delta,\mathrm{d}s\mathrm{d}u)$.

\begin{lemma}
Let $\varphi\in\mathscr{U}_0(H^1_0(\Omega))$ and $\phi=\mathscr{V}_\theta\varphi = (\varphi_1,\phi_2)$. We have
\begin{multline*}
\int_{B_\delta} (m_\theta^{-2}|h\partial_s\varphi_2|^2+[J'_\theta]^{-2}|h\partial_u\varphi_2|^2)m_\theta(s,u)J'_\theta(u)\mathrm{d}s\mathrm{d}u\\
=\int_{B_\delta} (m_\theta^{-2}|h\partial_s\phi_2|^2+[J'_\theta]^{-2}|h\partial_u\phi_2|^2+h^2V_\theta(s,u)|\phi_2|^2)\mathrm{d}s\mathrm{d}u+h^2\int_{-L}^L W_\theta(s)|\phi_2(s,\delta)|^2\mathrm{d}s\,,
\end{multline*}
 where, letting $X = m_\theta^{-1/2} (J_\theta')^{-1/2}$ we have
\[V_\theta= m_\theta^{-2} (\partial_s X)^2 + (J_\theta')^{-2}(\partial_uX)^2 -\partial_s \big( (m_\theta^{-2})X \partial_sX \big)-\partial_u \big((J'_\theta)^{-2}X\partial_uX \big)\]
and
\[
W_\theta(s)= (J_\theta')(s,\delta) X(s,\delta)(\partial_uX)(s,\delta)\,.
\]
\end{lemma}
\begin{proof}
This follows from two integrations by parts and from the fact  that $\phi_2(s,0) = 0$ \emph{a.e.} and that $J_\theta'$ is constant near $\partial B_\delta$. As we shall see later, the exact values of $V_\theta$ and $W_\theta$ are unimportant; we only note that they are smooth.
\end{proof}
These considerations lead to define the following quadratic form, in the ambient Hilbert space $L^2(\Omega\setminus T_\delta)\times L^2(B_\delta,\mathrm{d}s\mathrm{d}u)$,
 for all $\varphi\in\mathscr{U}_0(H^1_0(\Omega))$,
\begin{equation}  \label{eq.expLhat}
\begin{split}
\lefteqn{
	L_{h,\alpha,\theta}(\varphi,\varphi)
	=	\int_{\Omega\setminus T_\delta}
	\left( |\nabla\varphi_1|^2+e^{i\alpha}x_1|\varphi_1|^2 \right)
	\mathrm{d}x
	}
	\\
	&+\int_{B_\delta} (m_\theta^{-2}|h\partial_s\phi_2|^2+[J'_\theta]^{-2}|h\partial_u\phi_2|^2+(e^{i\alpha}\Gamma_1(s,J_\theta(u))+h^2V_\theta(s,u))|\phi_2|^2)\mathrm{d}s\mathrm{d}u\\
	&+h^2\int_{-L}^L W_\theta(s)|\phi_2(s,\delta)|^2\mathrm{d}s\,,
\end{split}
\end{equation}
where we recall that $\phi=\mathscr{V}_\theta\varphi = (\varphi_1,\phi_2)$.

We get the following lemma.
\begin{lemma}\label{lem.isoR}
	The operator associated with $L_{h,\alpha,\theta}$ is $\mathscr{L}_{h,\alpha,\theta}=\mathscr{V}_\theta\mathscr{U}_\theta\mathscr{L}_{h,\alpha}\mathscr{U}^{-1}_\theta\mathscr{V}_\theta^{-1}$. In particular, the spectrum of the operator $\mathscr{L}_{h,\alpha,\theta}$ is the same as that of $\mathscr{L}_{h,\alpha}$.
\end{lemma}

	\subsection{Complex deformation parameters}
According to our heuristic discussion, 	we would like to consider complex $\theta$. More precisely, we would like the family $(\mathscr{L}_{h,\alpha,\theta})_{\theta\in \Theta}$ to be analytic of type (B) in the sense of Kato, where $\Theta$ is a  connected open set containing $\theta=0$ and $\theta=-i\frac{\alpha}{3}$. First, we notice that the form domain $\mathscr{U}_0(H^1_0(\Omega))$ is independent of $\theta$ and that, for all $\varphi\in \mathscr{U}_0(H^1_0(\Omega))$,
\[\Theta\ni \theta\mapsto L_{h,\alpha,\theta}(\varphi,\varphi)\in\mathbb{C}\]
is analytic. Then, it is sufficient to
check that the form $L_{h,\alpha,\theta}$ is sectorial and closed on $\mathscr{U}_0(H^1_0(\Omega))$ for $\theta\in\Theta$.
	\begin{lemma}  \label{lem.jprimebelow}
Let $\theta_0>0$ and $\beta_0\in(0,\frac\pi4)$. For $\eta>0$  let us consider the rectangle $\Theta_\eta=(-\theta_0,\eta)+i(- \beta_0,\eta)$. Then, if $\eta$ and $\delta$ are  small enough,  there exists $c>0$ such that, for all $\theta\in\Theta_\eta$, and all $u\in(0,\delta)$,
\[\Re J'^{-2}_\theta(u)\geq c>0\,.\]
		\end{lemma}
	\begin{proof}
 Writing $\theta = \theta_1 + i\theta_2$ with $\theta_1, \theta_2 \in \R$
 and taking $u\in(0,\delta)$, we notice that	\Bk
\[ J'^{-2}_\theta(u)=|J'_\theta|^{-4}J'^{2}_{\overline{\theta}}=|J'_\theta|^{-4}(1+\theta_1 u\chi'-i\theta_2 u\chi')^2e^{2\theta_1\chi}e^{-2i\theta_2\chi}\,,\]
so that, by using that $\delta$ is small, we can write
\begin{equation} \label{eq.jprime} J'^{-2}_\theta(u)=|J'_\theta|^{-4}J'^{2}_{\overline{\theta}}=e^{2\theta_1\chi}|J'_\theta|^{-4}e^{-2i\arctan\left(\frac{\theta_2 u\chi'}{1+\theta_1 u\chi'}\right)-2i\theta_2\chi}  = e^{2\theta_1\chi}|J'_\theta|^{-4} e^{-iA(u,\theta)}\,,
\end{equation}
where the argument $A(u,\theta)$ is given by
\begin{equation} \label{eq.atheta}
A(u,\theta)=2\arctan\left(\frac{\theta_2 u\chi'}{1+\theta_1 u\chi'}\right)+2\theta_2\chi. \Bk
\end{equation}	

When $\theta_2$ is positive, we notice that,  for all $a>0$, by choosing $\eta$ small enough, and by using \eqref{eq.boundchi'}, we have, for all $\theta_2\in(0,\eta)$, $|A(u,\theta)|\leq a$.

 When $\theta_2$ is non-positive, namely $\theta_2\in(-\beta_0,0]$, we have, by using again \eqref{eq.boundchi'}, that
\[-2\beta_0\leq2\theta_2\leq 2\theta_2\chi\leq A(u,\theta)\leq\frac{2\theta_2 u\chi'}{1+\theta_1 u\chi'}\leq 2(1+2\eta)|\theta_2||u\chi'|\,.\]
Thus,
\[-\frac{\pi}{2}<-2\beta_0\leq A(u,\theta)\leq (1+4\eta)2\beta_0 <\frac{\pi}{2}\,.\]
 We therefore get the result by \eqref{eq.jprime}.
\end{proof}
		
		\begin{proposition}
There exist $c, C>0$ such that, for all $\varphi\in \mathscr{U}_0(H^1_0(\Omega))$,
	\[\Re L_{h,\alpha,\theta}(\varphi,\varphi)\geq c\|h\nabla\varphi\|^2_{E_0}-C\|\varphi\|^2_{E_0}\,.\]	
		\end{proposition}
		\begin{proof}
 Thanks to \eqref{eq.expLhat},  there exists $C>0$ such that the following holds: for all $\delta,\epsilon>0$, there exists $C_{\delta,\epsilon}>0$ such that, for all $\varphi\in \mathscr{U}_0(H^1_0(\Omega))$,
\begin{multline*}
\Re L_{h,\alpha,\theta}(\varphi,\varphi)\geq 	\int_{B_\delta}|h\partial_s\phi_2|^2\mathrm{d}s\mathrm{d}u+\int_{B_\delta}\Re J'^{-2}_\theta(u)|h\partial_u\phi_2|^2\mathrm{d}s\mathrm{d}u\\
+\|h\nabla\varphi_1\|^2_{L^2(\Omega\setminus T_\delta)}-C\delta\|h\nabla\phi_2\|^2_{L^2(B_\delta)}-C\|\varphi\|^2_{E_0}-C(\epsilon \|h\nabla\phi_2\|^2+C_{\delta,\epsilon} h^2\|\phi_2\|^2)\,,
\end{multline*}
where we used the classical estimate
\[\int_{-L}^L|\phi_2(s,\delta)|^2\mathrm{d}s\leq \epsilon\|\nabla\phi_2\|^2+C_{\delta,\epsilon}\|\phi_2\|^2\,.\]
 This concludes the proof.
\end{proof}

	\begin{corollary}\label{cor.iso}
		The operators $\mathscr{L}_{h,\alpha}$ and $\mathscr{L}_{h,\alpha,-i\alpha/3}$ are isospectral.
		\end{corollary}
	\begin{proof}
 It is a consequence of the analytic pertubation theory, upon observing that $-i\frac{\alpha}{3}\in\Theta_\eta$ since $-\beta<-\frac{\alpha}{3}$ is equivalent to $\alpha<3\beta$ which is satisfied for all $\alpha\in[0,\frac{3\pi}{4})$ as soon as $\beta$ is close enough to $\frac{\pi}{4}$.
		\end{proof}

	\section{Localization estimates} \label{sec.loestimates}
In virtue of Corollary \ref{cor.iso}, we now focus on the spectral analysis of $\mathscr{L}_{h,\alpha,-i\alpha/3}$, for which ellipticity properties are established in the present section.

\begin{definition} \label{def.m}  We denote $\mathscr{M}_{h,\alpha}=\mathscr{L}_{h,\alpha,-i\alpha/3}$ and by $M_{h,\alpha}$ the associate quadratic form on $E_0$, see \eqref{eq.expLhat} and Lemma \ref{lem.Etheta}.
\end{definition}
 In the following series of lemmas, we show bounds from below for
the potential part of $ e^{-i\beta}\mathscr{M}_{h,\alpha}$, where $\beta$ is introduced to correct a lack of coercivity of the real part when $\alpha$ is larger than $\frac\pi2$. These lemmas lead to Proposition \ref{prop.loch23}, which provides us with a precise semiclassical localization of the eigenfunctions of $\mathscr{M}_{h,\alpha}$.

The first lemma shows that the result of Lemma \ref{lem.jprimebelow} remains true if we insert $e^{-i\beta}$.

\begin{lemma}\label{lem.lbe-iaJ}
Consider $\alpha \geq 0$ and $\beta \in \R$ such that
\[\beta-\frac{2\alpha}{3}>-\frac\pi2\,,\quad \beta+\frac{2\alpha}{3}<\frac\pi2 \,.\]
Then, there exists $C(\alpha,\beta)>$0 such that, for all $u\in[0,\delta]$,
\[\Re(e^{-i\beta}J_\theta'^{-2}(u))\geq C(\alpha,\beta) \,.\]	
\end{lemma}

\begin{proof}
	We have
	\[\begin{split}
		\Re(e^{-i\beta}J_\theta'^{-2})&=\Re \left[e^{-i\beta+\frac{2i\alpha}{3}\chi}(1-i\frac\alpha3 u\chi')^{-2}\right]\\
	&=\left(1+\frac{\alpha^2}{9}(u\chi')^2\right)^{-1}
\Re\left[e^{-i\beta+\frac{2i\alpha}{3}\chi-2i\arctan(\frac\alpha3|u\chi'|)}\right]\\
\Bl &= \Re\left[e^{-i\beta-iA(u,-i\alpha/3)}\right]\,,
	\end{split}\]
where we recall $A(u,\theta)$ is defined in \eqref{eq.atheta}.
Note that, by using \eqref{eq.boundchi'}, we have, for $\epsilon>0$ small enough,
\begin{equation*}
-\frac\pi2<\beta-\frac{2\alpha}{3} \leq \beta-\frac{2\alpha}{3}\chi+2\arctan(\frac\alpha3|u\chi'|)
 \leq\beta+\frac{2\alpha}{3}(1+\epsilon)<\frac{\pi}{2}\,,
 \end{equation*}
so that $\Re (e^{-i\beta -iA(u,-i\alpha/3)})$ is uniformly bounded from below by a positive constant. This gives the result.
	\end{proof}
The following lemma is a preparation lemma in order to get the ellipticity of the electric potential in
$\mathscr{M}_{h,\alpha}$.

	\begin{lemma} \label{lem.defT}
 We let
\[	\mathcal{T} = \left\{(\alpha,\beta) \in \R^+\times \R  : \beta-\frac{2\alpha}{3}>-\frac\pi2\,,\quad \beta+\frac{2\alpha}{3}<\frac\pi2\,,\quad -\frac\pi2<\alpha-\beta<\frac\pi2\right\}\,.\]
Then
$
\sup\{ \alpha :  (\alpha, \beta) \in \mathcal{T} \} = \frac{3\pi}{5}\,.
$
Moreover, for all $\alpha\in\left[0,\frac{3\pi}{5}\right)$, we have $(\alpha, \frac{\pi}{10}) \in \mathcal{T}$. For all $\alpha \in [0, \frac{\pi}{2})$, we have $(\alpha, 0) \in \mathcal{T}$. We also have that $\left(\frac\pi2,\beta\right)\in\mathcal{T}$ for $\beta$ positive and small enough.
\end{lemma}

 \begin{proof}
These estimates follow from straightforward computations,
 which are conveniently supported by drawing a picture.
 We leave the details to the reader.
\end{proof}

\begin{lemma}\label{lem.confinement}
 Assume that $\alpha\in\left[0,\frac{3\pi}{5}\right)$ and consider $\beta$ such that $(\alpha,\beta) \in \mathcal{T}$. There exist $s_0, \delta_0,c>0$ such that the following holds. For all $\delta\in(0,\delta_0)$, for all $(s,u)\in[-L,L)\times(0,\delta)$, if $|s|\geq s_0$, then
	\[	\Re(e^{i(\alpha-\beta)}\Gamma_1(s,J_\theta(u)))\geq c\,,\]
	and, if $|s|\leq s_0$,
	\[	\Re(e^{i(\alpha-\beta)}\Gamma_1(s,J_\theta(u)))\geq c(u+s^2)\,.\]
	\end{lemma}
\begin{proof}
 Consider $(\alpha,\beta)$ as in the statement.	We have
	\begin{equation*}
	\Re(e^{i(\alpha-\beta)}\Gamma_1(s,J_\theta(u)))\\
	=\Re(e^{i(\alpha-\beta)} \Gamma_1(s,ue^{-i\frac\alpha3\chi}))\,.
	\end{equation*}
	From \eqref{def.gamma}, we have
	\[\Gamma_1(s,e^{-i\frac\alpha3\chi}u)=\gamma_1(s)-ue^{-i\frac\alpha3\chi}n_1(s)\,,\]
and thus, by using the Taylor expansion with respect to $s$ near $0$ and Assumption \ref{geoassumption0}, we get
	\[\Gamma_1(s,e^{-i\frac\alpha3\chi}u)=ue^{-i\frac\alpha3\chi}+\frac{\kappa_0}{2} s^2+\mathscr{O}(|s|^3+us^2)\,.\]
Therefore,
\begin{equation*}
	\Re(e^{i(\alpha-\beta)}\Gamma_1(s,J_\theta(u)))\\
=u\cos\left(\alpha-\beta-\frac{\alpha}{3}\chi\right)+\frac{ \kappa_0 }{2}s^2\cos(\alpha-\beta)+\mathscr{O}(|s|^3+us^2)\,.
\end{equation*}
Since $(\alpha,\beta) \in \mathcal{T}$,
\[-\frac\pi2<\frac{2\alpha}{3}-\beta\leq \alpha-\beta-\frac{\alpha}{3}\chi<\frac\pi2\,
\quad \textrm{ and } -\frac\pi2<\alpha-\beta<\frac\pi2\,
.\]
Therefore, there exist $s_0,\delta_0>0$ such that, for all $\delta\in(0,\delta_0)$, there exists $c>0$ such that, for all $s$ such  that $|s|\leq s_0$ and all $u\in(0,\delta)$, we have
\begin{equation*}
		\Re(e^{i(\alpha-\beta)}\Gamma_1(s,J_\theta(u)))
	\geq c(u+s^2)\,.
\end{equation*}

 Let us now study the case when $|s| \geq s_0$. We first introduce
\[B(s,u) =\arctan\left(\frac{u\sin(\frac\alpha3\chi(u))n_1(s)}{\gamma_1(s)-u\cos(\frac\alpha3\chi(u))n_1(s)}\right)\,,\]
so that
\[\Gamma_1(s,e^{-i\frac\alpha3\chi}u)=|\gamma_1(s)-ue^{-i\frac\alpha3\chi}n_1(s)| e^{iB(s,u)}.\,\]
We notice that when $|s|\geq s_0$, and choosing $\delta$ small enough, we have $\gamma_1(s)-u\cos(\frac\alpha3\chi)n_1(s)\geq c_0>0$ uniformly. This implies that
\begin{equation*}
\begin{split}
	\Re(e^{i(\alpha-\beta)}\Gamma_1(s,J_\theta(u)))
&\geq |\gamma_1(s)-ue^{-i\frac\alpha3\chi}n_1(s)|\cos(\alpha-\beta +B(s,u))\, \\
& \geq c_0 \cos(\alpha-\beta +B(s,u))\,.
\end{split}
\end{equation*}
From the expression of $B(s,u)$, we deduce that, with a possibly smaller $c$, we have, for all $|s|\geq s_0$,
\[	\Re(e^{i(\alpha-\beta)}\Gamma_1(s,J_\theta(u)))\geq c>0\,.\]
 This completes the proof.
	\end{proof}

The following proposition gives Agmon type localization estimates for some eigenfunctions of $\mathscr{M}_{h,\alpha}$.

\begin{proposition}\label{prop.loch23}
 Let $\alpha \in [0, \frac{3\pi}{5})$ and consider suitable parameters $\beta$, $\delta_0$ introduced in Lemma \ref{lem.confinement}. Then for any $M>0$ and $0<\delta<\delta_0$, there exists $h_0,C>0$ such that for all $h\in(0,h_0)$, all eigenvalue $\lambda$ (of $\mathscr{M}_{h,\alpha}$) such that $\Re (e^{-i\beta}\lambda)\leq Mh^{\frac23}$ and all associated eigenfunction $\varphi=(\varphi_1,\phi_2)$, we have
\begin{equation}\label{eq.AgmonL2}
\int_{\Omega\setminus T_\delta}e^{2|x|/h^{\frac23}}|\varphi_1|^2\mathrm{d}x+\int_{B_\delta}e^{2|\Gamma(s,u)|/h^{\frac23}}|\phi_2|^2\mathrm{d}s\mathrm{d}u\leq C\|\varphi\|^2_{E_0}\,,
\end{equation}
and
\[\int_{\Omega\setminus T_\delta}e^{2|x|/h^{\frac23}}|h\nabla\varphi_1|^2\mathrm{d}x+\int_{B_\delta}e^{2|\Gamma(s,u)|/h^{\frac23}}|h\nabla_{s,u}\phi_2|^2\mathrm{d}s\mathrm{d}u\leq Ch^{\frac23}\|\varphi\|^2_{E_0}\,.\]
\end{proposition}
\begin{proof}
The proof essentially follows from the classical Agmon estimates.
Considering $\tilde\varphi=(e^{2|x|/h^{2/3}}\varphi_1,e^{2|\Gamma(s,u)|/h^{2/3}}\phi_2)$,  we see that $\tilde\varphi\in\mathscr{U}_0(H^1_0(\Omega))$. We have then
\begin{equation}\label{eq.Mphiphi}
\langle\mathscr{M}_{h,\alpha}\varphi,\tilde\varphi\rangle=\lambda\left(\|e^{|x|/h^{2/3}}\varphi_1\|^2_{\Omega\setminus T_\delta}+\|e^{|\Gamma(s,u)|/h^{2/3}}\phi_2\|^2_{T_\delta}\right)\,,
\end{equation}
and, recalling Definition \ref{def.m}, we can write that
\begin{equation}\label{eq.III}
e^{-i\beta}M_{h,\alpha}\left(\varphi,\tilde\varphi\right)=\mathrm{I}+\mathrm{II}+\mathrm{III}\,,
\end{equation}
with
\begin{equation*}
\mathrm{I}=e^{-i\beta}h^2\langle \nabla\varphi_1,\nabla(e^{2|x|/h^{2/3}}\varphi_1)\rangle_{\Omega\setminus T_\delta}+e^{i(\alpha-\beta)}\int_{\Omega\setminus T_\delta} x_1|e^{|x|/h^{2/3}}\varphi_1|^2\mathrm{d}x\,,
\end{equation*}
\begin{multline*}
\mathrm{II}=\int_{B_\delta} \big[e^{-i\beta}m_\theta^{-2}h^2\partial_s\phi_2\partial_s(e^{2|\Gamma|/h^{2/3}}\overline{\phi}_2)+e^{-i\beta}[J'_\theta]^{-2}h^2\partial_u\phi_2\partial_u(e^{2|\Gamma|/h^{2/3}}\overline{\phi}_2)\\
+e^{i(\alpha-\beta)}\Gamma_1(s,J_\theta(u))|e^{|\Gamma|/h^{2/3}}\phi_2|^2\big]\mathrm{d}s\mathrm{d}u\,
\end{multline*}
and
\[\mathrm{III}=h^2\int_{B_\delta}V_\theta(s,u)|e^{|\Gamma|/h^{2/3}}\phi_2|^2\mathrm{d}s\mathrm{d}u+h^2\int_{-L}^LW_\theta(s)|e^{|\Gamma|/h^{2/3}}\phi_2(s,\delta)|^2\mathrm{d}s\,.\]
Let us now bound the real part of $\mathrm{I}$ and $\mathrm{II}$ from below. We have
\[\Re\mathrm{I}=\Re e^{-i\beta}h^2\langle \nabla\varphi_1,\nabla(e^{2|x|/h^{2/3}}\varphi_1)\rangle_{\Omega\setminus T_\delta}+\cos(\alpha-\beta)\|\sqrt{x_1}e^{|x|/h^{2/3}}\varphi_1\|^2\,.\]
Then, with the chain rule, we get
\begin{equation*}
	\begin{split}
	&h^2\langle \nabla\varphi_1,\nabla(e^{2|x|/h^{2/3}}\varphi_1)\rangle_{\Omega\setminus T_\delta}\\
&	=h^2\langle e^{|x|/h^{2/3}}\nabla\varphi_1,\nabla(e^{|x|/h^{2/3}}\varphi_1)\rangle_{\Omega\setminus T_\delta}+h^{\frac43}\langle \nabla\varphi_1,(\nabla |x|)(e^{2|x|/h^{2/3}}\varphi_1)\rangle_{\Omega\setminus T_\delta}\\
&	=h^2\|\nabla(e^{|x|/h^{2/3}}\varphi_1)\|^2_{\Omega\setminus T_\delta}+\mathscr{O}(h^{\frac43})\|e^{|x|/h^{2/3}}\varphi_1\|\|\nabla(e^{|x|/h^{2/3}}\varphi_1)\|\\
	&\qquad \qquad +h^{\frac43}\langle \nabla\varphi_1,\nabla |x|(e^{2|x|/h^{2/3}}\varphi_1)\rangle_{\Omega\setminus T_\delta}\\
& 	=h^2\|\nabla(e^{|x|/h^{2/3}}\varphi_1)\|^2_{\Omega\setminus T_\delta}+\mathscr{O}(h^{\frac43})\|e^{|x|/h^{2/3}}\varphi_1\|\|\nabla(e^{|x|/h^{2/3}}\varphi_1)\|+\mathscr{O}(h^{\frac23})\|e^{|x|/h^{2/3}}\varphi_1\|^2\,.
	\end{split}
	\end{equation*}
Thus, with the Young inequality, we deduce that
\begin{multline}\label{eq.ReI}
\Re\mathrm{I}\geq c(\alpha,\beta)\left(h^2\|\nabla(e^{|x|/h^{2/3}}\varphi_1)\|^2_{\Omega\setminus T_\delta}+\|\sqrt{x_1}e^{|x|/h^{2/3}}\varphi_1\|^2_{\Omega\setminus T_\delta}\right)\\
-Ch^{\frac23}\|e^{|x|/h^{2/3}}\varphi_1\|^2_{\Omega\setminus T_\delta}\,.
\end{multline}
 Note that we did not use any integration by parts in the last computation, so that no boundary term appears. We proceed rather similarly to see that
\begin{equation*}
	\begin{split}
	&\int_{B_\delta}m_\theta^{-2}h^2\partial_s\phi_2\partial_s(e^{2|\Gamma|/h^{2/3}}\overline{\phi}_2)
\mathrm{d}s\mathrm{d}u\\	
&= \int_{B_\delta}m_\theta^{-2}h^2|\partial_s(e^{|\Gamma|/h^{2/3}}\phi_2)|^2\mathrm{d}s\mathrm{d}u
        +\mathscr{O}(h^{1/3})\|e^{|\Gamma|/h^{2/3}}\phi_2\|\|h\partial_s(e^{|\Gamma|/h^{2/3}}\phi_2)\|\\
& \qquad +\mathscr{O}(h^{\frac23})\|e^{|\Gamma|/h^{2/3}}\phi_2\|^2\\
&=	\int_{B_\delta}h^2|\partial_s(e^{|\Gamma|/h^{2/3}}\phi_2)|^2\mathrm{d}s\mathrm{d}u
     +\mathscr{O}(h^{1/3})\|e^{|\Gamma|/h^{2/3}}\phi_2\|\|h\partial_s(e^{|\Gamma|/h^{2/3}}\phi_2)\|\\
& \qquad +\mathscr{O}(\delta)\|h\partial_s(e^{|\Gamma|/h^{2/3}}\phi_2)\|^2_{B_\delta}
+\mathscr{O}(h^{\frac23})\|e^{|\Gamma|/h^{2/3}}\phi_2\|^2\,,
	\end{split}
	\end{equation*}
 where we used $m_\theta = 1+ \mathcal{O}(\delta)$ from \eqref{eq.mtheta}.
	In the same way, we get
	\begin{equation*}
	\begin{split}	&\int_{B_\delta}[J'_\theta]^{-2}h^2\partial_u\phi_2\partial_u(e^{2|\Gamma|/h^{2/3}}\overline{\phi}_2)\\	& =\int_{B_\delta}[J_\theta'(u)]^{-2}h^2|\partial_u(e^{|\Gamma|/h^{2/3}}\phi_2)|^2\mathrm{d}s\mathrm{d}u
+\mathscr{O}(h^{4/3})\|e^{|\Gamma|/h^{2/3}}\phi_2\|\|\partial_u(e^{|\Gamma|/h^{2/3}}\phi_2)\|\\
		&\qquad +\mathscr{O}(\delta)h^2\|\partial_u(e^{|\Gamma|/h^{2/3}}\phi_2)\|^2_{B_\delta}
+\mathscr{O}(h^{\frac23})\|e^{|\Gamma|/h^{2/3}}\phi_2\|^2\,.
		\end{split}
		\end {equation*}
		We deduce that
\begin{equation*}
\begin{split}
\Re \mathrm{II}\geq
& \int_{B_\delta}\cos(\alpha-\beta)h^2|\partial_s(e^{|\Gamma|/h^{2/3}}\phi_2)|^2\mathrm{d}s\mathrm{d}u\\	
&+ \int_{B_\delta}\Re(e^{-i\beta}J_\theta'(u)^{-2})h^2|\partial_u(e^{|\Gamma|/h^{2/3}}\phi_2)|^2
\mathrm{d}s\mathrm{d}u\\
&+\int_{B_\delta}\Re(e^{i(\alpha-\beta)}\Gamma_1(s,J_\theta(u)))|e^{|\Gamma|/h^{2/3}}\phi_2|^2
\mathrm{d}s\mathrm{d}u\\
&+\mathscr{O}(h^{4/3})\|e^{|\Gamma|/h^{2/3}}\phi_2\|\|\nabla(e^{|\Gamma|/h^{2/3}}\phi_2)\|
+\mathscr{O}(\delta)h^2\|\nabla(e^{|\Gamma|/h^{2/3}}\phi_2)\|^2_{B_\delta}\\
&+\mathscr{O}(h^{\frac23})\|e^{|\Gamma|/h^{2/3}}\phi_2\|^2\,.
\end{split}
\end{equation*}
With Lemma \ref{lem.lbe-iaJ}, we get, for some $c(\alpha,\beta)>0$,
\begin{equation*}
\begin{split}
\Re \mathrm{II}\geq& c(\alpha,\beta)\int_{B_\delta}h^2|\nabla(e^{|\Gamma|/h^{2/3}}\phi_2)|^2
\mathrm{d}s\mathrm{d}u\\
&+\int_{B_\delta}\Re(e^{i(\alpha-\beta)}\Gamma_1(s,J_\theta(u)))|e^{|\Gamma|/h^{2/3}}\phi_2|^2
\mathrm{d}s\mathrm{d}u\\
&+\mathscr{O}(h^{4/3})\|e^{|\Gamma|/h^{2/3}}\phi_2\|\|\nabla(e^{|\Gamma|/h^{2/3}}\phi_2)\|
+\mathscr{O}(\delta)h^2\|\nabla(e^{|\Gamma|/h^{2/3}}\phi_2)\|^2_{B_\delta}\\
			&+\mathscr{O}(h^{\frac23})\|e^{|\Gamma|/h^{2/3}}\phi_2\|^2\,,
		\end{split}
	\end{equation*}
and thus, perhaps after changing the value of $c(\alpha,\beta)>0$,
			\begin{equation}\label{eq.ReII}
	\begin{split}
		\Re \mathrm{II}\geq&  c(\alpha,\beta)\int_{B_\delta}|h\nabla(e^{|\Gamma|/h^{2/3}}\phi_2)|^2
\mathrm{d}s\mathrm{d}u
-Ch^{\frac23}\|e^{|\Gamma|/h^{2/3}}\phi_2\|^2\\
		&+\int_{B_\delta}\Re(e^{i(\alpha-\beta)}\Gamma_1(s,J_\theta(u)))|e^{|\Gamma|/h^{2/3}}\phi_2|^2
\mathrm{d}s\mathrm{d}u\,.
	\end{split}
\end{equation}
Moreover, by using a $H^{\frac12}$-trace theorem, we get that, for all $\epsilon\in(0,1)$,  there exists $C_\epsilon$ such that
\begin{equation}\label{eq.ReIII}
\Re\mathrm{III}\geq -C_\epsilon h^2\|e^{|\Gamma|/h^{2/3}}\phi_2\|^2-\epsilon\left\|h\nabla \left(e^{|\Gamma|/h^{2/3}}\right)\phi_2\right\|^2_{B_\delta}.
\end{equation}
 Now ,we can come back to \eqref{eq.Mphiphi} and \eqref{eq.III} and we deduce, from \eqref{eq.ReI}, \eqref{eq.ReII}, \eqref{eq.ReIII}, and $\Re (e^{-i\beta}\lambda)\leq Mh^{\frac23}$, that
\begin{multline*}
c\|\sqrt{x_1}e^{|x|/h^{2/3}}\varphi_1\|^2_{\Omega\setminus T_\delta}+\int_{B_\delta}\Re(e^{i(\alpha-\beta)}\Gamma_1(s,J_\theta(u)))|e^{|\Gamma|/h^{2/3}}\phi_2|^2
\mathrm{d}s\mathrm{d}u\\
\leq  M  h^{\frac23}\left(\|e^{|x|/h^{2/3}}\varphi_1\|^2_{\Omega\setminus T_\delta}+\|e^{|\Gamma(s,u)|/h^{2/3}}\phi_2\|^2_{T_\delta}\right)\,.
\end{multline*}

To conclude, we now split the integral into two parts according to the decomposition
\[B_\delta=\{(s,u)\in B_\delta : u+s^2\leq Rh^{\frac23}\}\cup B_\delta^{\mathrm{far}}\,.\]
where
\[B_\delta^{\mathrm{far}}\{(s,u)\in B_\delta : u+s^2  >  Rh^{\frac23}\}\,.\]
The rest of the proof follows  from the usual manipulations \textit{\`a la} Agmon. Indeed, thanks to Lemma \ref{lem.confinement}, there exists $C_M$ such that
\[(cRh^{\frac23}-C_Mh^{\frac23}) \big(\|e^{|x|/h^{2/3}}\varphi_1\|_{(\Omega\setminus T_\delta)\times B^{\mathrm{far}}_\delta} + \|e^{|\Gamma(s,u)|/h^{\frac23}}\phi_2\|^2_{(\Omega\setminus T_\delta)\times B^{\mathrm{far}}_\delta}\big)\leq \tilde Ch^{\frac23}\|\varphi\|^2\,,\] \Bk
where we used that the exponential is bounded on $\{(s,u)\in B_\delta : u+s^2\leq Rh^{\frac23}\}$. The estimate \eqref{eq.AgmonL2} follows by choosing $R$ large enough. Then, the estimate of the gradient follows by gathering  \eqref{eq.Mphiphi}, \eqref{eq.III}, \eqref{eq.ReI}, \eqref{eq.ReII}, \eqref{eq.ReIII} and by using \eqref{eq.AgmonL2}.

\end{proof}

	\section{Spectral analysis}  \label{sec.spectralanalysis}

The aim of this section is to prove the following three propositions, which imply Theorem \ref{thm.0} (since $\mathscr{M}_{h,\alpha}$ is isospectral to $\mathscr{L}_{h,\alpha}$, see Lemma \ref{cor.iso}).

\begin{proposition}[Rough localization of the spectrum]\label{prop.lowerbound}
 Let $\alpha\in[0,\frac{\pi}{2})$ and consider $M>0$. There exist $C,h_0>0$ such that, for all $h\in(0,h_0)$,
\begin{multline}
\mathrm{sp}(\mathscr{L}_{h,\alpha}) \cap \{  z\in\mathbb{C} \ : \  \Re(z) < M h^{\frac23} \} \\
\subset \big\{ z\in\mathbb{C} \ : \  \Re z\geq z_1h^{\frac23}\cos\left(2\alpha/3\right)-Ch^{\frac43}), \\
 0\leq \Im z\leq (\sin\alpha)\max \{ x_1 : x \in \overline{\Omega} \big\}\,.
\end{multline}	
\end{proposition}

\begin{proposition}[Refined localization of the spectrum]\label{prop.locspec} Consider $\alpha\in\left[0,\frac{3\pi}{5}\right)$  and $R>0$ with $R \not\in (2\N-1)\sqrt{\frac{\kappa_0}{2}}$. Then there exist $h_0>0$ and $N \in \N$ such that, for all $h\in(0,h_0)$,
	\begin{equation}\label{eq.locspec}
	\mathrm{sp}(\mathscr{M}_{h,\alpha})\cap D(z_1e^{\frac{2i\alpha}{3}}h^{\frac23},Rh)\subset\bigcup_{n=1}^{N} D(\mu_n(h,\alpha),h^{\frac32-2\eta})\,,
	\end{equation}
	with
\begin{equation} \label{eq.munhalpha} \mu_n(h,\alpha)=h^{\frac23}e^{2i\alpha/3}z_1+(2n-1)he^{i\alpha/2}\sqrt{\frac{\kappa_0}{2}}\,.
\end{equation}
	Moreover, for all $n\in\{1,\ldots, N\}$, the Riesz projector
	\[\Pi_{n,h}:=\frac{1}{2i\pi}\int_{\mathscr{C}_{n,h}}(z-\mathscr{M}_{h,\alpha})^{-1}\mathrm{d}z\,,\quad\mbox{ where } \mathscr{C}_{n,h}=\partial D(\mu_n(h,\alpha),h^{\frac32-2\eta})\,,\]
	is of rank at most one.
	
	\end{proposition}

	\begin{proposition}[Existence of the spectrum]\label{prop.finale}
Consider $\alpha\in\left[0,\frac{3\pi}{5}\right)$. There exists $h_0>0$ such that, for all $h\in(0,h_0)$, the rank of $\Pi_{n,h}$ is exactly one.
	\end{proposition}

Proposition \ref{prop.lowerbound} is proved in Section \ref{prop.lowerbound}. Propositions \ref{prop.locspec} and \ref{prop.finale} are proved in Section~\ref{sec.4.3}. Proposition \ref{prop.locspec} is a consequence of the analysis in Sections \ref{sec.4.2} and \ref{sec.4.1} where, by inserting an appropriate quasimode in the Riesz projector, we can prove that $\Pi_{n,h}$ is not zero.

\subsection{Rough localization of the spectrum}\label{sec.roughloc}
 In this section, we prove Proposition \ref{prop.lowerbound}.

 Let us first note that, in a first naive approach, when $\alpha \in [0,\frac{\pi}{2})$,
\[\Re\langle\mathscr{L}_{h,\alpha}\psi,\psi\rangle\geq 0\,,\quad \Im\langle\mathscr{L}_{h,\alpha}\psi,\psi\rangle=\sin\alpha\int_ {\Omega}x_1|\psi|^2\mathrm{d}x\,.\]
%\end{proof}
This gives the estimate on the imaginary part of the spectrum, and it remains to refine the estimate on the real part. For this, we consider \Bk a smooth cutoff function $\chi_h$ in the form
\[\chi_h(s,u)=\chi(h^{-\frac13+\eta}s,h^{-\frac23+\eta}u)\,,\]	
 where $\eta>0$ and we let
\begin{equation}\label{eq.cut}
	\phi^{\mathrm{cut}}_2=\chi_h\phi_2\,.
\end{equation}
Consider now an eigenfunction $\varphi$ of $\mathscr{M}_{h,\alpha}$ associated with an eigenvalue $\lambda$ such that $\Re
\lambda\leq Mh^{\frac23}$. From \eqref{eq.expLhat}, we get
\[\begin{split}
	\Re\lambda\|\varphi\|^2 & =\Re L_{h,\alpha,\theta}(\varphi,\varphi)\\
	\Bl & \geq \int_{B_\delta} \Big(\Re(m_\theta^{-2})|h\partial_s\phi_2|^2+\Re([J'_\theta]^{-2})|h\partial_u\phi_2|^2 \\
& \qquad \qquad \Bk +\Re \left(e^{i\alpha}\Gamma_1(s,J_\theta(u))|\phi_2|^2\right)\Big)\mathrm{d}s\mathrm{d}u
	-Ch^2\|\phi_2\|^2\,,
\end{split}\]
where we used the trace theorem and Proposition \ref{prop.loch23} to control the boundary term. Then, with a Taylor expansion  in the expressions \eqref{eq.mtheta} of $m_\theta$ and \eqref{eq.jprimetheta} of $J'_{\theta}$ and with Proposition \ref{prop.loch23}, we  get
\[\begin{split}
\Re\lambda\|\varphi\|^2& =\Re L_{h,\alpha,\theta}(\varphi,\varphi)\\
&\geq \int_{B_\delta} |h\partial_s\phi^{\mathrm{cut}}_2|^2+\cos\left(\frac{2\alpha}{3}\right)
\left(|h\partial_u\phi^{\mathrm{cut}}_2|^2+u|\phi_2^{\mathrm{cut}}|^2\right)
\mathrm{d}s\mathrm{d}u -Ch^{\frac43}\|\phi_2\|^2\,.
\end{split}\]
From the min-max theorem applied to the real Airy operator, we deduce that
\[\Re \lambda\|\varphi\|^2\geq z_1h^{\frac23}\cos\left(\frac{2\alpha}{3}\right)\|\phi_2^{\mathrm{cut}}\|^2-Ch^{\frac43}\|\phi_2\|^2\,.\]
By using again Proposition \ref{prop.loch23}, we infer that
\[\Re \lambda\geq z_1h^{\frac23}\cos\left(\frac{2\alpha}{3}\right)-Ch^{\frac43}\,.\]
This proves Proposition \ref{prop.lowerbound}.

\begin{remark}  \label{rem.interm}
Let us mention what happens in the case $\alpha \in (\frac{\pi}{2}, \frac{3\pi}{5})$. In that case we do not even \emph{a priori} have $\Re\langle\mathscr{M}_{h,\alpha}\psi,\psi\rangle\geq 0$
because of the electric potential. Choosing $\beta$ such that $(\alpha,\beta) \in \mathcal{T}$ where $\mathcal{T}$ is defined in Lemma \ref{lem.defT},
we can still use the localization properties proved in Section
\ref{sec.loestimates}. In particular, working with $e^{-i\beta} \mathscr{M}_{h,\alpha}$ instead of
$\mathscr{M}_{h,\alpha}$ and using Proposition \ref{prop.loch23}, as well as Taylor expansions of $m_\theta$ and of $J'_{\theta}$, we get that, for any eigenvalue $\lambda$ of $\mathcal{L}_{h,\alpha}$, we have
\[
\Re e^{-i\beta}\lambda\geq z_1h^{\frac23}\cos\left(\frac{2\alpha}{3}-\beta\right)-Ch^{\frac43}\,.
\]
\end{remark}

\subsection{Resolvent estimates}\label{sec.4.2}
To prepare the proof of Proposition \ref{prop.locspec}, let us describe the spectrum and resolvent of our model operator
\begin{equation}\label{eq.Nhalpha}
\mathscr{N}_{h,\alpha}=e^{\frac{2i\alpha}{3}}(h^2D_u^2+u)+h^2D_s^2+e^{i\alpha}\frac{ \kappa_0 s^2}{2}\,,
\end{equation}
where $D = -i\partial$.
\begin{proposition}\label{prop.resolventmodel}
	Let $R>0$ with $R \not\in (2\N-1)\sqrt{\frac{\kappa_0}{2}}$. There exist $C, h_0>0$ and $N \in\mathbb{N}$ such that the following holds. The spectrum of $\mathscr{N}_{h,\alpha}$ in $D(z_1e^{\frac{2i\alpha}{3}}h^{\frac23},Rh)$ is made of $N$ eigenvalues of algebraic multiplicity one, which are the $(\mu_n(h,\alpha))_{1\leq n\leq N}$ as given in \eqref{eq.munhalpha}.  Moreover, for all $z\in D(z_1e^{\frac{2i\alpha}{3}}h^{\frac23},Rh)$ such that $z\notin\{\mu_n(h,\alpha)\,, n\in\{1,\ldots, N \}\}$, we have
	\[\|(z-\mathscr{N}_{h,\alpha})^{-1}\|\leq Ch^{-\frac23}+\frac{C}{\mathrm{dist}(z,\mathrm{sp}(\mathscr{N}_{h,\alpha}))}\,.\]
\end{proposition}
\begin{proof}
	Note that $\mathscr{N}_{h,\alpha}$ is the sum of two decoupled operators, an Airy operator and a harmonic oscillator. It has compact resolvent and its spectrum and eigenfunctions are completely known. Consider
	\begin{equation}\label{eq.Psimn}
	\Psi_{m,n,h}(s,u)=h^{-\frac13-\frac14}\mathsf{Ai}(h^{-\frac23}u-z_m)
f_n(h^{-\frac12}e^{i\alpha/2}\sqrt{\frac{\kappa_0}{2}}s)\,,
	\end{equation}
	where $\mathsf{Ai}$ is the usual Airy function (and $z_m$ its $m$-th positive zero) and $f_n$ the $n$-th normalized Hermite function. We have
	\[\mathscr{N}_{h,\alpha}\Psi_{m,n,h}=\left(z_mh^{\frac23}e^{\frac{2i\alpha}{3}}
+(2n-1)h\sqrt{\frac{\kappa_0}{2}}\right)\Psi_{m,n,h}\,.\]
	Moreover, there are no other eigenvalues and they are all of multiplicity one. Indeed, by analytic dilation, we see that $\mathscr{N}_{h,\alpha}$ is isospectral to the normal operator
	\[e^{\frac{2i\alpha}{3}}(h^2D_u^2+u)+e^{\frac{i\alpha}{2}}\left(h^2D_s^2+\frac{k_0 s^2}{2}\right)\,.\]
	Let us now turn to the estimate of the resolvent in the disk $D(z_1e^{\frac{2i\alpha}{3}}h^{\frac23},Rh)$. Consider
	\[z=z_1e^{\frac{2i\alpha}{3}}h^{\frac23}+\zeta h\,,\quad \zeta\in D(0,R)\,,\]
	with $\zeta$ avoiding the numbers $e^{\frac{i\alpha}{2}}(2n-1)\sqrt{\frac{k_0}{2}}$.
	We have
	\[\mathscr{N}_{h,\alpha}-z=e^{\frac{2i\alpha}{3}}(h^2D_u^2+u-z_1 h^{\frac23})+h^2D_s^2+e^{i\alpha}\frac{k_0 s^2}{2}-\zeta h\,,\]
	and also
	\[e^{-\frac{i\alpha}{2}}(\mathscr{N}_{h,\alpha}-z)=e^{\frac{i\alpha}{6}}(h^2D_u^2+u-z_1 h^{\frac23})+e^{-\frac{i\alpha}{2}}\left(h^2D_s^2+e^{i\alpha}\frac{k_0 s^2}{2}-\zeta h\right)\,.\]
	Let us denote $g_h$ the  (explicit) positive normalized groundstate of $h^2D_u^2+u$ and consider the orthogonal projection
	\[\mathfrak{P}_h\psi(s,u)=\langle g_h,\psi(s,\cdot)\rangle_{L^2(\mathbb{R}_+)} g_h(u)\,.\]
	We have
	\begin{equation} \label{eq:borninforth}\begin{split}
	\Re\langle e^{-\frac{i\alpha}{2}}(\mathscr{N}_{h,\alpha}-z)(\mathrm{Id}-\mathfrak{P}_h)\psi,\psi\rangle&\geq \left(h^{\frac23}\cos(\alpha/6)(z_2-z_1)-Rh\right)\|(\mathrm{Id}-\mathfrak{P}_h)\psi\|^2\\
  &\geq ch^{\frac23}\|(\mathrm{Id}-\mathfrak{P}_h)\psi\|^2\,,
	\end{split}
\end{equation}
 which implies that the restriction of $\mathscr{N}_{h,\alpha}-z$ to $\textrm{Ker}(\mathfrak{P}_h)$ is injective and therefore bijective since $\mathscr{N}_{h,\alpha}$ is Fredholm of index $0$.
Moreover, we have the orthogonal decomposition
\[e^{-\frac{i\alpha}{2}}(\mathscr{N}_{h,\alpha}-z)
=e^{-\frac{i\alpha}{2}}(\mathscr{N}_{h,\alpha}-z)(\mathrm{Id}-\mathfrak{P}_h)
+e^{-\frac{i\alpha}{2}}\left(h^2D_s^2+e^{i\alpha}\frac{ \kappa_0  s^2}{2}-\zeta h\right)\mathfrak{P}_h\,.\]
 From \eqref{eq:borninforth}, the fact that $\zeta$ avoids the numbers $e^{\frac{i\alpha}{2}}(2n-1)\sqrt{\frac{k_0}{2}}$, we get that $e^{-\frac{i\alpha}{2}}(\mathscr{N}_{h,\alpha}-z)$ is bijective and that its inverse is given by
\begin{equation*}
\begin{split}
\big(e^{-\frac{i\alpha}{2}}(\mathscr{N}_{h,\alpha}-z)\big)^{-1}
= & e^{\frac{i\alpha}{2}}\big((\mathscr{N}_{h,\alpha}-z)\big|_{\textrm{Ker}(\mathfrak{P}_h)}\big)^{-1}
(\mathrm{Id}-\mathfrak{P}_h) \\
& +e^{\frac{i\alpha}{2}}\big(h^2D_s^2+e^{i\alpha}\frac{k_0 s^2}{2}-\zeta h\big)^{-1}\mathfrak{P}_h\,.
\end{split}
\end{equation*}
 Let us notice that there exists $C>0$ such that for all $\zeta\in D(0,R)$ avoiding the numbers $e^{\frac{i\alpha}{2}}(2n-1)\sqrt{\frac{k_0}{2}}$, we have  \Bk
\[\left\|\left(h^2D_s^2+e^{i\alpha}\frac{k_0 s^2}{2}-\zeta h\right)^{-1}\right\|\leq \frac{C}{\mathrm{dist}(\mathrm{sp}(h^2D_s^2+e^{i\alpha}\frac{k_0 s^2}{2}),\zeta h)}\,.\]
Thus, from the above orthogonal decomposition, we deduce that
\[\|[e^{-\frac{i\alpha}{2}}(\mathscr{N}_{h,\alpha}-z)]^{-1}\|\leq (ch^{\frac23})^{-1}+\frac{C}{\mathrm{dist}(\mathrm{sp}(h^2D_s^2+e^{i\alpha}\frac{k_0 s^2}{2}),\zeta h)}\,.\]
This concludes the proof.
\end{proof}

\subsection{Quasimodes and localization estimates}\label{sec.4.1}
	If  $\varphi =(\varphi_1,\phi_2)$ is a normalized eigenfunction of $\mathscr{M}_{h,\alpha}$ associated with $\lambda\in D(z_1e^{\frac{2i\alpha}{3}}h^{\frac23},Rh)$, we have in particular
\begin{equation}
	\mathcal{M}_{h,\alpha}\phi_2=\lambda\phi_2\,,
\end{equation}
with $\mathcal{M}_{h,\alpha}$ denoting the formal operator
\begin{equation} \label{eq.formalM}
\mathcal{M}_{h,\alpha}=hD_s(m_\theta^{-2})hD_s
+hD_u(J'_\theta)^{-2}hD_u+e^{i\alpha}\Gamma_1(s,J_\theta(u))\,.
\end{equation}

This section is devoted to the proof of the following proposition. We recall that the model operator $\mathscr{N}_{h,\alpha}$ is defined in \eqref{eq.Nhalpha} and that $\phi^{\mathrm{cut}}_2$ is defined as a truncation of $\phi_2$ in \eqref{eq.cut}.
\begin{proposition}\label{prop.quasimode}
	We have
	\[(\mathscr{N}_{h,\alpha}-\lambda)\phi_2^{\mathrm{cut}}
=\mathscr{O}(h^{\frac32-3\eta})\|\phi^{\mathrm{cut}}_2\|\,.\]	
\end{proposition}

 This proposition essentially says that $\phi^{\mathrm{cut}}_2$ is a good quasimode for our model operator $\mathscr{N}_{h,\alpha}$. The proof will be done in several steps including elliptic estimates and a refined localization in $s$.

\subsubsection{Preliminary estimates}
 Recall that $\chi_h(s,u)=\chi(h^{-\frac13+\eta}s,h^{-\frac23+\eta}u)$ with $\chi$ a smooth cutoff function and $\eta$ fixed and small, and that $\phi^{\mathrm{cut}}_2=\chi_h\phi_2$.
	\begin{lemma}\label{lem.cut}
We have
\[\mathcal{M}_{h,\alpha}\phi^{\mathrm{cut}}_2=\lambda\phi^{\mathrm{cut}}_2+r_h\,,\]
with
\[r_h=	[\mathcal{M}_{h,\alpha},\chi_h]\phi_2\,,\]		
 where the commutator is given by
\begin{multline}\label{eq.comMchih}
 [\mathcal{M}_{h,\alpha},\chi_h]\phi =\big( hD_s(m^{-2}_\theta (hD_s\chi_h))\big) \phi +2m_\theta^{-2}(hD_s\chi_h)(hD_s \phi)\\
	+\big( hD_u((J'_\theta)^{-2} (hD_u\chi_h))\big) \phi +2(J'_\theta)^{-2}(hD_u\chi_h)(hD_u\phi)\,.
\end{multline}
\Bk
Moreover,
\[\|r_h\|=\mathscr{O}(h^\infty)\|\phi^{\mathrm{cut}}_2\|\,.\]
		\end{lemma}
	\begin{proof}
The expression \eqref{eq.comMchih} follows from a straightforward computation. The estimate $\|r_h\|=\mathscr{O}(h^\infty)$ is a consequence of Proposition \ref{prop.loch23} and support considerations.
		\end{proof}
	
\begin{remark}
With a straightforward computation, we can check  that, for all $k,\ell\in\mathbb{N}$, $\|D_s^kD_u^\ell r_h\|=\mathscr{O}(h^\infty)$.	
\end{remark}
	
		\begin{lemma}\label{lem.hnabla23}
We have
\[\|hD_s\phi^{\mathrm{cut}}_2\|^2+\|hD_u\phi^{\mathrm{cut}}_2\|^2\leq  Ch^{\frac2
3} \|\phi^{\mathrm{cut}}_2\|^2\,.\]		
	\end{lemma}
	\begin{proof}
 Due to the Dirichlet boundary condition of $\phi_2$ on the external part of $\partial B_\delta$ and the cutoff $\chi_h$, the function $\phi_2^{\mathrm{cut}}$ satisfies the Dirichlet condition on $\partial B_\delta$.

 Let us now choose
$\beta$ such that $(\alpha,\beta) \in \mathcal{T}$ as defined in Lemma \ref{lem.defT}.  With an integration by parts using the expression of $e^{-i\beta}\mathcal{M}_{h,\alpha}$, we get that,
for some $c>0$,
\[\Re\left(e^{-i\beta}\langle\mathcal{M}_{h,\alpha}\phi_2^{\mathrm{cut}},\phi_2^{\mathrm{cut}}\rangle\right)\geq c(\|hD_s \phi^{\mathrm{cut}}_2\|^2+\|hD_u \phi^{\mathrm{cut}}_2\|^2)\,.\]
Then, by Lemma \ref{lem.cut}, we have that
\[\left|\Re\left(e^{-i\beta}\langle\mathcal{M}_{h,\alpha}\phi_2^{\mathrm{cut}},\phi_2^{\mathrm{cut}}\rangle\right)\right|\leq Ch^{\frac23}\|\phi_2^{\mathrm{cut}}\|^2\,.\]
The conclusion follows.
	\end{proof}

 For further use, we check now that we even have a control of higher order derivatives.
\begin{lemma}\label{lem.hnabla243}
	We have
	\[\|(hD_s)^2\phi^{\mathrm{cut}}_2\|^{ 2}+\|(hD_u)^2\phi^{\mathrm{cut}}_2\|^{ 2} \leq Ch^{ \frac43}\|\phi^{\mathrm{cut}}_2\|^2\,.\]		
\end{lemma}	
\begin{proof}
By Lemma \ref{lem.cut}, we have
	\[\langle hD_s(\mathcal{M}_{h,\alpha}\phi^{\mathrm{cut}}_2), hD_s\phi_2^{\mathrm{cut}}\rangle=\lambda \|hD_s \phi^{\mathrm{cut}}_2\|^2+\langle r_h, (hD_s)^2\phi_2^{\mathrm{cut}}\rangle\,.\]
	Thus, after computing a commutator, we get
		\begin{multline*}
		\langle \mathcal{M}_{h,\alpha}(hD_s\phi^{\mathrm{cut}}_2), hD_s\phi_2^{\mathrm{cut}}\rangle\\
		=\lambda \|hD_s \phi^{\mathrm{cut}}_2\|^2+\langle r_h, (hD_s)^2\phi_2^{\mathrm{cut}}\rangle+h\mathscr{O}(\|\phi^{\mathrm {cut}}_2\|\|hD_s\phi_2^{\mathrm{cut}}\|)
		+h^2\mathscr{O}(\|hD_s\phi_2^{\mathrm{cut}}\|^2)\\
		+h\mathscr{O}(\|(hD_s)^2\phi_2^{\mathrm{cut}}\|\|hD_s\phi_2^{\mathrm {cut}}\|)\,.
		\end{multline*}
		Then, we use Lemma \ref{lem.hnabla23} and we proceed as in its the proof to get
		\[\|(hD_s)^2\phi_2^{\mathrm{cut}}\|^2+\|(hD_u)(hD_s\phi_2^{\mathrm{cut}})\|^2\leq Ch^{\frac43}\|\phi_2^{\mathrm{cut}}\|^2+Ch^{\frac43}\|(hD_s)^2\phi_2^{\mathrm{cut}}\|^2\,.\]
		We proceed in the same way to get the control of $(hD_u)^2\phi_2^{\mathrm{cut}}$.
		The conclusion follows.
	\end{proof}	
 Let us consider the following intermediate operator,
\[\mathcal{N}_{h,\alpha}=e^{\frac{2i\alpha}{3}}(h^2D^2_u+u)+h^2D_s^2+e^{i\alpha}\gamma_1(s)\,,\]	
	which differs from $\mathscr{N}_{h,\alpha}$ in \eqref{eq.Nhalpha} only through its potential part.
\begin{proposition}\label{prop.quasimode43}
We have
\[\mathcal{N}_{h,\alpha}\phi_2^{\mathrm{cut}}=\lambda\phi_2^{\mathrm{cut}}+R_h\,,\]
with
\[
R_h=r_h+u e^{2i\alpha/3}(\mathbf{n}_1(s)-1)\phi_2^{\mathrm{cut}}-(hD_sm_\theta^{-2})hD_s\phi_2^{\mathrm{cut}}
+(1-m_\theta^{-2})(hD_s)^2\phi_2^{\mathrm{cut}}
\]
and
\[\|R_h\|=\mathscr{O}(h^{\frac43-\eta}) \|\phi^{\mathrm{cut}}_2\|\,.\]	
	\end{proposition}
\begin{proof}
It is sufficient to use Lemma \ref{lem.cut} and the explicit expression of $\mathcal{M}_{h,\alpha}$. The estimate of $R_h$ follows from support considerations and Lemma \ref{lem.hnabla243} (note also that $\mathbf{n}_1(s)-1=\mathscr{O}(s^2)$).
\end{proof}

\subsubsection{Refined localization in $s$}
 Thanks to the Agmon estimates in Proposition \ref{prop.loch23}, we have proved so far a localization of order $h^{\frac13}$ in the variable $s$. We improve this in the following proposition.
\begin{proposition}\label{prop.locs}
Consider $\chi_0\in\mathscr{C}^\infty_0(\mathbb{R})$ equal to $1$ in a neighborhood of $0$ and $\eta>0$. Then,
\begin{equation}\label{eq.locs}
\phi_2^{\mathrm{cut}}=\chi_0(h^{-\frac12+\eta}s)\phi_2^{\mathrm{cut}}+\mathscr{O}(h^{\infty})\,,
\end{equation}
where the remainder is estimated in $H^1$-norm.

Moreover, for all $(\alpha_1,\alpha_2,\alpha_3)\in\mathbb{N}^3$,
\begin{equation}\label{eq.hDsopt}
\|s^{\alpha_1}(hD_s)^{\alpha_2}(hD_u)^{\alpha_3}\phi_2^{\mathrm{cut}}\|\leq Ch^{\frac{\alpha_3}{3}+\frac{\alpha_1+\alpha_2}{2}}\|\phi_2^{\mathrm{cut}}\|\,.	
\end{equation}
\end{proposition}
\begin{proof}
Let us consider $\chi_{\mathrm{far},h}(s)=\chi_{\mathrm{far}}(h^{-\frac12+\eta}s)$, where $\chi_{\mathrm{far}}$ is supported away from $0$ and equal to $1$ away form a compact set.  Then, we use Proposition \ref{prop.quasimode43} to get that
\begin{multline}\label{eq.N-lambdachifar}
%\begin{split}
\langle (\mathcal{N}_{h,\alpha}-\lambda)
(\chi_{\mathrm{far},h}\phi_{2}^{\mathrm{cut}}),\chi_{\mathrm{far},h}\phi_{2}^{\mathrm{cut}}\rangle
 =\langle \chi_{\mathrm{far},h}R_h,\chi_{\mathrm{far},h}\phi_{2}^{\mathrm{cut}}\rangle \\
 \qquad +\langle[\mathcal{N}_{h,\alpha},\chi_{\mathrm{far},h}]\phi_{2}^{\mathrm{cut}},
\chi_{\mathrm{far},h}\phi_{2}^{\mathrm{cut}}\rangle\,.
%\end{split}
\end{multline}
 Notice that $-\alpha/2+2\alpha/3 = \alpha/6$ and that, for all $\Psi\in H^1_0(\mathbb{R}_+)$,
\begin{equation}\label{eq.minoairy}
\cos(\alpha/6)\langle (h^2D^2_u+u- z_1h^{\frac23} )\Psi,\Psi\rangle\geq 0\,.
\end{equation}
This, combined with the fact that $\gamma_1(s)$ is bounded from below by $h^{1-2\eta}$ on the support of
$\chi_{\mathrm{far},h}$, implies that, for some $c>0$,
\begin{multline*}
\Re\langle(e^{-i\alpha/2}\big( (\mathcal{N}_{h,\alpha}-\lambda)(\chi_{\mathrm{far},h}\phi_{2}^{\mathrm{cut}}),
\chi_{\mathrm{far},h}\phi_{2}^{\mathrm{cut}}\big)\rangle\\
\geq ch^{1-2\eta}\|\chi_{\mathrm{far},h}\phi_2^{\mathrm{cut}}\|^2
+c\|hD_s(\chi_{\mathrm{far},h}\phi_{2}^{\mathrm{cut}})\|^2\,.
\end{multline*}
From the support properties of $\phi_{2}^{\mathrm{cut}}$ (see before Lemma \ref{lem.cut}) and Proposition \ref{prop.quasimode43}, we have
\begin{equation*}
	\begin{split}
\big|\langle \chi_{\mathrm{far},h}R_h,\chi_{\mathrm{far},h}\phi_{2}^{\mathrm{cut}}\rangle\big|\leq&
\mathscr{O}(h^\infty)\|\phi_2^{\mathrm{cut}}\|^2
+Ch^{\frac43-2\eta}\|{\chi}_{\mathrm{far},h}\phi_2^{\mathrm{cut}}\|^2\\
&+Ch\|hD_s(\chi_{\mathrm{far},h}\phi_{2}^{\mathrm{cut}})
\|\|\chi_{\mathrm{far},h}\phi_{2}^{\mathrm{cut}}\|\\
&+|\langle(1-m_\theta^{-2})\chi_{\mathrm{far},h}(hD_s)^2\phi_2^{\mathrm{cut}},
\chi_{\mathrm{far},h}\phi_{2}^{\mathrm{cut}}\rangle|\,,
\end{split}
\end{equation*}
where we used that $u(\mathbf{n}_1(s)-1)=\mathscr{O}(us^2)$.

Using a commutator in the last term of the previous expression, we also have
\begin{multline*}
|\langle(1-m_\theta^{-2})\chi_{\mathrm{far},h}(hD_s)^2\phi_2^{\mathrm{cut}},
\chi_{\mathrm{far},h}\phi_{2}^{\mathrm{cut}}\rangle|\leq\ Ch^{\frac23-\eta}\|(hD_s)(\chi_{\mathrm{far},h}\phi_2^{\mathrm{cut}})\|^2\\
+Ch^{\frac43-3\eta}\|\underline{\chi}_{\mathrm{far},h}\phi_2^{\mathrm{cut}}\|^2
+Ch\|hD_s(\chi_{\mathrm{far},h}\phi_{2}^{\mathrm{cut}})
\|\|\chi_{\mathrm{far},h}\phi_{2}^{\mathrm{cut}}\|+\mathscr{O}(h^\infty)\|\phi_2^{\mathrm{cut}}\|^2\,,
\end{multline*}
where $\underline{\chi}_{\mathrm{far}}$ has the same properties as $\chi_{\mathrm{far}}$ and is such that $\underline{\chi}_{\mathrm{far}}\chi_{\mathrm{far}}=\chi_{\mathrm{far}}$.

In a similar way, we get that
\begin{multline*}
	\langle[\mathcal{N}_{h,\alpha},\chi_{\mathrm{far},h}]\phi_{2}^{\mathrm{cut}},
\chi_{\mathrm{far},h}\phi_{2}^{\mathrm{cut}}\rangle|\leq Ch^{\frac43-2\eta}\|\underline{\chi}_{\mathrm{far},h}\phi_2^{\mathrm{cut}}\|^2\\
	+Ch\|hD_s(\chi_{\mathrm{far},h}\phi_{2}^{\mathrm{cut}})\|\|\chi_{\mathrm{far},h}\phi_{2}^{\mathrm{cut}}\|+\mathscr{O}(h^\infty)\|\phi_2^{\mathrm{cut}}\|^2\,.\end{multline*}
It follows that
\[\frac{c}{2}\|hD_{s}(\chi_{\mathrm{far},h}\phi_{2}^{\mathrm{cut}})\|^2+ch^{1-2\eta}\|\chi_{\mathrm{far},h}\phi_2^{\mathrm{cut}}\|^2\leq Ch^{\frac43-3\eta}\|\underline{\chi}_{\mathrm{far},h}\phi_2^{\mathrm{cut}}\|^2+\mathscr{O}(h^\infty)\|\phi_2^{\mathrm{cut}}\|^2\,.\]
By choosing $\eta$ small enough, and by using an induction argument, we get that, for all $N\in\mathbb{N}$,
\[\chi_{\mathrm{far},h}\phi_2^{\mathrm{cut}}=\mathscr{O}(h^N)\|\phi_2^{\mathrm{cut}}\|\,,\]
in $H^1$-norm with respect to $s$. Coming back to \eqref{eq.N-lambdachifar}, we also get the control of $hD_u$. This gives \eqref{eq.locs}.

Let us now turn to \eqref{eq.hDsopt}, which are better estimates than those in Lemmas \ref{lem.hnabla23} and \ref{lem.hnabla243}. We again use Proposition \ref{prop.quasimode43} and see that
\[\Re e^{-i\alpha/2}\langle (\mathcal{N}_{h,\alpha}-\lambda)\phi_{2}^{\mathrm{cut}},\phi_{2}^{\mathrm{cut}}\rangle\leq|\langle R_h,\phi_{2}^{\mathrm{cut}}\rangle|\,.\]
By using \eqref{eq.minoairy}, we get that
\begin{equation*}
\|s\phi_2^{\mathrm{cut}}\|^2+\|hD_s\phi_2^{\mathrm{cut}}\|^2\leq Ch\|\phi_2^{\mathrm{cut}}\|^2+Ch^{\frac23-\eta}\|(hD_s)\phi_2^{\mathrm{cut}}\|^2+Ch\|(hD_s)\phi_2^{\mathrm{cut}}\|\|\phi_2^{\mathrm{cut}}\|\,,
\end{equation*}
and the estimate for $|(\alpha_1,\alpha_2,\alpha_3)|=1$ follows from the Young inequality. Let us now deal with $|(\alpha_1,\alpha_2,\alpha_3)|=2$. We have
\begin{equation}\label{eq.hDshDu}
\Re e^{-i\alpha/2}\langle (hD_s)(\mathcal{N}_{h,\alpha}-\lambda)\phi_{2}^{\mathrm{cut}},hD_s\phi_{2}^{\mathrm{cut}}\rangle\leq|\langle R_h,(hD_s)^2\phi_{2}^{\mathrm{cut}}\rangle|\,.
\end{equation}
Thus,
\begin{multline*}
c\|shD_s\phi_2^{\mathrm{cut}}\|^2+\|(hD_s)^2\phi_2^{\mathrm{cut}}\|^2\leq Ch\|hD_s\phi_2^{\mathrm{cut}}\|^2+Ch\|s\phi_2^{\mathrm{cut}}\|\|hD_s\phi_2^{\mathrm{cut}}\|\\
+|\langle R_h,(hD_s)^2\phi_{2}^{\mathrm{cut}}\rangle|\,.
\end{multline*}
By using the Young inequality and Proposition \ref{prop.quasimode43} to deal with the last term, we get
\[\|shD_s\phi_2^{\mathrm{cut}}\|^2+\|(hD_s)^2\phi_2^{\mathrm{cut}}\|^2\leq Ch^2\|\phi_2^{\mathrm{cut}}\|^2\,.\]
Coming back to \eqref{eq.hDshDu}, we also get
\[\|hD_u(hD_s\phi_2^{\mathrm{cut}})\|^2\leq Ch^{\frac23}\|hD_s\phi_2^{\mathrm{cut}}\|^2+Ch^{2+\frac43-\eta}\|\phi_2^{\mathrm{cut}}\|^2\leq Ch^{\frac23+1}\|\phi_2^{\mathrm{cut}}\|^2\,.\]
To get the control of $s^2$, it is sufficient to notice that
\[\Re e^{-i\alpha/2}\langle s(\mathcal{N}_{h,\alpha}-\lambda)\phi_{2}^{\mathrm{cut}},s\phi_{2}^{\mathrm{cut}}\rangle\leq|\langle R_h,s^2\phi_{2}^{\mathrm{cut}}\rangle|\,,\]
and to estimate again a commutator. This concludes the case $|(\alpha_1,\alpha_2,\alpha_3)|=2$. The proof of \eqref{eq.hDsopt} for general $(\alpha_1,\alpha_2,\alpha_3)$  follows then by induction using the same method.
\end{proof}

\subsubsection{Proof of Proposition \ref{prop.quasimode}}
 We are now in position to complete the proof of Proposition \ref{prop.quasimode}, namely that
$\phi_2^{\mathrm{cut}}$ is indeed a good quasimode for $\mathscr{N}_{h,\alpha}$. For this, we consider
an operator $P_h$ defined through:
\begin{equation}\label{eq.MNP}
\mathcal{M}_{h,\alpha}=\mathscr{N}_{h,\alpha}+P_h\,.
\end{equation}
 We check that $P_h$ can be written in the following way:
\[P_h=hur_{h,1}(s,u)hD_s+ur_{h,2}(s,u)(hD_s)^2+r_{h,3}(s,u) us^2+r_{h,4}(s,u)s^3+r_{h,5}(s,u)\,,\]

where the remainders $r_{h,j}$ all belong to $S_{\mathbb{R}^2}(1)$ and $r_{h,5}$ with support avoiding a fixed neighborhood of $(0,0)$. We recall that $S_{\mathbb{R}^2}(1)=\{a\in\mathscr{C}^\infty(\R^2) : \forall\alpha\in\N^2\,,\quad \exists C_\alpha>0 : \forall x\in\R^2 : |\partial^\alpha a(x)|\leq C_\alpha  \}$. For the other terms, we use the support property in the variable $u$ and Proposition \ref{prop.locs} to get
$$
\|P_h\phi_2^{\mathrm{cut}}\|=\mathscr{O}(h^{\frac32-3\eta})\|\phi_2^{\mathrm{cut}}\|\,.
$$
The conclusion follows.

\subsection{Quasimodes for $(\mathscr{N}_{h,\alpha}-\lambda)^2$} \label{sec.quasimodes2}
 Since we are in a non-selfadjoint context, the algebraic and geometric dimension associated with a given eigenvalue $\lambda$ may differ. For further use, we now deal with localization estimates similar to the ones in Proposition \ref{prop.quasimode}, but in the case of generalized eigenfunctions. For this, let us consider such a $\lambda\in D(z_1e^{\frac{2i\alpha}{3}}h^{\frac23},Rh)$ associated with $\varphi = (\varphi_1,\phi_2) \in\ker(\mathcal{M}_{h,\alpha}-\lambda)^2$ with $\|\varphi\|=1$ and such that $\varphi\notin\ker(\mathcal{M}_{h,\alpha}-\lambda)$ (if it exists). We still denote
\[\phi^{\mathrm{cut}}_2=\chi_h\phi_2\,.\]
The following proposition states that $\phi^{\mathrm{cut}}_2$ is a generalized quasimode of $\mathscr{N}_{h,\alpha}$. Its proof is the object of the following two sections.
\begin{proposition}\label{prop.quasimode2}
	We have
	\[(\mathscr{N}_{h,\alpha}-\lambda)^2\phi_2^{\mathrm{cut}}
=\mathscr{O}(h^{\frac{13}{6}})\|\phi^{\mathrm{cut}}_2\|\,.\]	
\end{proposition}

\subsubsection{Localization estimates}
The function $\phi_2^{\mathrm{cut}}$ satisfies the same localization estimates as in the previous section.  Let us explain this. We have
\[(\mathscr{M}_{h,\alpha}-\lambda)^2\varphi=0\,,\]
for $\varphi$ normalized and $(\mathscr{M}_{h,\alpha}-\lambda)\varphi=f\neq 0$. We have $(\mathscr{M}_{h,\alpha}-\lambda)f=0$ and thus $f$ satisfies the estimates of the previous section.  For instance, we have
\[f_2^{\mathrm{cut}}=\chi(h^{-\frac12+\eta}s)f_2^{\mathrm{cut}}+\mathscr{O}(h^{\infty})\|f\|\,,\]
 where $f_2^{\mathrm{cut}}$ is defined without ambiguity and satifies,  from Proposition \ref{prop.locs},
\begin{equation}\label{eq.estimatesf}
\|s^{\alpha_1}(hD_s)^{\alpha_2}(hD_u)^{\alpha_3}f_2^{\mathrm{cut}}\|\leq Ch^{\frac{\alpha_3}{3}+\frac{\alpha_1+\alpha_2}{2}}\|f_2^{\mathrm{cut}}\|\,
\end{equation}
for all $(\alpha_1,\alpha_2,\alpha_3) \in \N^3$.
Coming back to the eigenvalue equation, this implies that
\begin{equation}\label{eq.quasiairy}
(h^2D^2_u+u-z_1h^{\frac23})f_2^{\mathrm{cut}}=\mathscr{O}(h)\|f_2^{\mathrm{cut}}\|\,.
\end{equation}
We can easily adapt the proof of the Agmon estimates given in Proposition \ref{prop.loch23}  with the right-hand side $f$ to get
\begin{equation}\label{eq.Agmonfa}
	\int_{\Omega\setminus T_\delta}e^{2|x|/h^{\frac23}}|\varphi_1|^2\mathrm{d}x+\int_{B_\delta}e^{2|\Gamma(s,u)|/h^{\frac23}}|\phi_2|^2\mathrm{d}s\mathrm{d}u\leq C\|\varphi\|^2_{E_0}+Ch^{-\frac23}\|f\|^2_{E_0}\,,
\end{equation}
and
\begin{equation}\label{eq.Agmonfb}
	\int_{\Omega\setminus T_\delta}e^{2|x|/h^{\frac23}}|h\nabla\varphi_1|^2\mathrm{d}x+\int_{B_\delta}e^{2|\Gamma(s,u)|/h^{\frac23}}|h\nabla_{s,u}\phi_2|^2\mathrm{d}s\mathrm{d}u\leq Ch^{\frac23}\|\varphi\|^2_{E_0}+C\|f\|^2_{E_0}\,.
	\end{equation}
These estimates imply that
\begin{equation}\label{eq.phi2f2}
(\mathcal{M}_{h,\alpha}-\lambda)\phi_2^{\mathrm{cut}}=f_2^{\mathrm{cut}}+r_h\,,
\end{equation}
where $r_h$ has the same expression as in Lemma \ref{lem.cut} and satisfies
\[r_h=\mathscr{O}(h^\infty)(\|f\|+\|\varphi\|)\,.\]
In the following proposition we prove that, in fact, $f$ is small compared to $\phi$. This estimate is reminiscent of the famous Caccioppoli estimates
(see the original article \cite{C51}, and, for instance, the article \cite{I03} or the book \cite[Section 5.4.1]{AIM09}), since it allows us to control the derivatives of $\phi$ with $\phi$.
\begin{proposition}\label{prop.controlf}
	We have
	\begin{equation}\label{eq.controlf}
	\|f\|\leq Ch\|\varphi\|\leq \tilde Ch\|\phi_2^{\mathrm{cut}}\|\,,
		\end{equation}
	and, for all $(\alpha_1,\alpha_2) \in\mathbb{N}^2$,
	\begin{equation}\label{eq.hDshDuphi}
	\|(hD_u)^{\alpha_1}(hD_s)^{\alpha_2}\phi_2^{\mathrm{cut}}\|\leq Ch^{\frac{\alpha_1}{3}}h^{\frac{\alpha_2}{2}}\|\varphi\|\,.
	\end{equation}
	\end{proposition}
\begin{proof}
Let us start by noticing that, from \eqref{eq.phi2f2}, we have
\[\Re e^{-i\alpha/2}\langle (\mathcal{M}_{h,\alpha}-\lambda)\phi_2^{\mathrm{cut}},\phi_2^{\mathrm{cut}}\rangle\leq \|f_2^{\mathrm{cut}}\|\|\phi_2^{\mathrm{cut}}\|+\mathscr{O}(h^\infty)(\|f\|+\|\varphi\|)\,,\]
and thus, with \eqref{eq.MNP} and localization estimates, we get
\begin{multline}\label{eq.f-alpha=1}
\cos\left(\frac{\alpha}{6}\right)q_{\mathrm{Ai},h}(\phi_2^{\mathrm{cut}})+c\|hD_s\phi_2^{\mathrm{cut}}\|^2+c\|s\phi_2^{\mathrm{cut}}\|^2\leq Ch\|\phi_2^{\mathrm{cut}}\|^2+ \|f_2^{\mathrm{cut}}\|\|\phi_2^{\mathrm{cut}}\|\\
+\mathscr{O}(h^\infty)(\|f\|+\|\varphi\|)\,,
\end{multline}
where, for all $\Psi\in B_0^1(\R_+):=\{\Psi\in H^1_0(\R_+) : \sqrt{u}\Psi\in L^2(\R_+) \}$,
\[q_{\mathrm{Ai},h}(\Psi)=\|hD_u\Psi\|^2+\int_{\mathbb{R}_+\times\mathbb{R}}(u-z_1 h^{\frac23})|\Psi|^2\mathrm{d}u\mathrm{d}s\geq 0\,.\]
We will denote the corresponding operator by $\mathscr{A}_h$ and we observe that $\phi_2^{\mathrm {cut}} $ and $f_2^{\mathrm {cut}}$ belong to its domain.

At this stage, we still have to control $f$.
By using \eqref{eq.phi2f2}, an integration by parts and \eqref{eq.estimatesf}, we have
 \begin{equation}
\begin{split}
 \langle hD_s(\mathcal{M}_{h,\alpha}-\lambda)\phi_2^{\mathrm{cut}},hD_s\phi_2^{\mathrm{cut}}\rangle
& =  \langle (hD_s)^2 f_2^{\mathrm{cut}}, \phi_2^{\mathrm{cut}}\rangle
+ \langle r_h,(hD_s)^2 \phi_2^{\mathrm{cut}}\rangle \\
& \leq
Ch\|f_2^{\mathrm{cut}}\|\|\phi^{\mathrm{cut}}_2\| + \mathscr{O}(h^\infty)\|(hD_s)^2\phi_2^{\mathrm{cut}}\|(\|f\|+\|\varphi\|)\,.
\end{split}
\end{equation}
Multiplying by $e^{-i\alpha/2}$, taking the real part and estimating commutators give
\begin{multline}\label{eq.f-alpha=2b}
\|(hD_s)^2\phi_2^{\mathrm{cut}}\|^2+q_{\mathrm{Ai},h}(hD_s\phi_2^{\mathrm{cut}})\\
\leq Ch\|hD_s\phi_2^{\mathrm{cut}}\|^2	+Ch\|f_2^{\mathrm{cut}}\|\|\phi^{\mathrm{cut}}_2\|
+Ch\|(hD_s)^2\phi_2^{\mathrm{cut}}\|\|hD_s\phi_2^{\mathrm {cut}}\|\\
+\mathscr{O}(h^\infty)\|(hD_s)^2\phi_2^{\mathrm{cut}}\|(\|f\|+\|\varphi\|)\,,
\end{multline}
where we note that $hD_s \phi_2^{\mathrm {cut}}$ is also in $B^1_0(\R_+)$.

Proceeding in the same way, we find
\begin{multline*}
\Re \left( e^{-i\alpha/2}\langle (\mathcal{M}_{h,\alpha}-\lambda)\phi_2^{\mathrm{cut}},\mathscr{A}_h\phi_2^{\mathrm{cut}}\rangle \right)
\leq \mathscr{O}(h^\infty)\|\mathscr{A}_h\phi_2^{\mathrm{cut}}\|(\|f\|+\|\varphi\|)+C\|\mathscr{A}_h f_2^{\mathrm{cut}}\|\|\phi^{\mathrm{cut}}_2\|\,,
\end{multline*}
and thus, with \eqref{eq.quasiairy},
\begin{multline*}
\Re \left( e^{-i\alpha/2}\langle (\mathcal{M}_{h,\alpha}-\lambda)\phi_2^{\mathrm{cut}},\mathscr{A}_h\phi_2^{\mathrm{cut}}\rangle \right)
\leq \mathscr{O}(h^\infty)\|\mathscr{A}_h\phi_2^{\mathrm{cut}}\|(\|f\|+\|\varphi\|)+Ch\|f_2^{\mathrm{cut}}\|\|\phi^{\mathrm{cut}}_2\|\,.
\end{multline*}
Therefore, estimating similarly commutators,
\begin{equation*}
\begin{split}
q_{\mathrm{Ai},h}(hD_s\phi_2^{\mathrm{cut}})+\|\mathscr{A}_h\phi_2^{\mathrm{cut}}\|^2\leq Chq_{\mathrm{Ai},h}(\phi_2^{\mathrm{cut}})
+\mathscr{O}(h^\infty)\|\mathscr{A}_h\phi_2^{\mathrm{cut}}\|(\|f\|+\|\varphi\|)\\
	+Ch\|f_2^{\mathrm{cut}}\|\|\phi^{\mathrm{cut}}_2\|
+Ch^{\frac43-2\eta}\|hD_s\phi_2^{\mathrm{cut}}\|^2
+Ch\|hD_s\phi^{\mathrm{cut}}_2\|\|hD_u(hD_s\phi_2^{\mathrm{cut}})\|\,.
\end{split}
\end{equation*}
Note that
 \[\|hD_u(hD_s\phi_2^{\mathrm{cut}})^2\|\leq  q_{\mathrm{Ai},h}(hD_s\phi_2^{\mathrm{cut}}) +Ch^{\frac23}\|hD_s\phi_2^{\mathrm{cut}}\|^2\,,\]
so that, using the definition of $ q_{\mathrm{Ai},h}$, we get
\begin{multline}\label{eq.hDsA12}
	 q_{\mathrm{Ai},h}(hD_s\phi_2^{\mathrm{cut}})  +\|\mathscr{A}_h\phi_2^{\mathrm{cut}}\|^2\leq Ch  q_{\mathrm{Ai},h}(\phi_2^{\mathrm{cut}})
+\mathscr{O}(h^\infty)\|\mathscr{A}_h\phi_2^{\mathrm{cut}}\|(\|f\|+\|\varphi\|)\\
	+Ch\|f_2^{\mathrm{cut}}\|\|\phi^{\mathrm{cut}}_2\|
+Ch^{\frac43-2\eta}\|hD_s\phi_2^{\mathrm{cut}}\|^2\,.
\end{multline}
Thus, with \eqref{eq.f-alpha=2b}, \eqref{eq.hDsA12} and the Young inequality, we get, for some constant $c>0$,
\begin{multline*}
c\|\mathscr{A}_h\phi_2^{\mathrm{cut}}\|^2+c\|(hD_s)^2\phi_2^{\mathrm{cut}}\|^2+q_{\mathrm{Ai},h}(hD_s\phi_2^{\mathrm{cut}})\\
\leq Ch^{2}\|\phi_2^{\mathrm{cut}}\|^2+\mathscr{O}(h^\infty)(\|f\|^2+\|\varphi\|^2)+Ch\|f_2^{\mathrm{cut}}\|\|\phi^{\mathrm{cut}}_2\|\,.
\end{multline*}
Recalling \eqref{eq.phi2f2} to bound the right-hand side, we deduce that
\begin{equation}\label{eq.f-alpha=2}
c\|\mathscr{A}_h\phi_2^{\mathrm{cut}}\|^2+c\|(hD_s)^2\phi_2^{\mathrm{cut}}\|^2
+q_{\mathrm{Ai},h}(hD_s\phi_2^{\mathrm{cut}})
\leq Ch^{2}\|\phi_2^{\mathrm{cut}}\|^2\,.
\end{equation}
Again with \eqref{eq.phi2f2}, this implies that
\[\|f\|\leq C\|f_2^{\mathrm{cut}}\|\leq Ch\|\phi_2^{\mathrm{cut}}\| \,,\]
which gives \eqref{eq.controlf}.

With \eqref{eq.f-alpha=1}, \eqref{eq.f-alpha=2b}, and \eqref{eq.f-alpha=2}, we get \eqref{eq.hDshDuphi} for $|(\alpha_1,\alpha_2)|\in\{0,1,2\}$. The control of the higher powers can be obtained by induction and similar estimates.
\end{proof}

\begin{proposition}[Localisation in $s$]\label{prop.locoptimal}
	Consider $\chi_0\in\mathscr{C}^\infty_0(\mathbb{R})$ equal to $1$ in a neighborhood of $0$ and $\eta>0$. Then, in $H^1$-norm, we have
	\begin{equation}\label{eq.locs2bis}
		\phi_2^{\mathrm{cut}}=\chi_0(h^{-\frac12+\eta}s)\phi_2^{\mathrm{cut}}+\mathscr{O}(h^{\infty})\,.
	\end{equation}
	Moreover, for all $(\alpha_1,\alpha_2,\alpha_3)\in\mathbb{N}^3$,
	\begin{equation}\label{eq.hDsoptbis}
		\|s^{\alpha_1}(hD_s)^{\alpha_2}(hD_u)^{\alpha_3}\phi_2^{\mathrm{cut}}\|\leq Ch^{\frac{\alpha_3}{3}+\frac{\alpha_1+\alpha_2}{2}}\|\phi_2^{\mathrm{cut}}\|\,.	
	\end{equation}
	\end{proposition}

\begin{proof}
The proof is similar to that of Proposition \ref{prop.locs} since we can write
\begin{multline*}\langle (\mathcal{N}_{h,\alpha}-\lambda)(\chi_{\mathrm{far},h}\phi_{2}^{\mathrm{cut}}),\chi_{\mathrm{far},h}\phi_{2}^{\mathrm{cut}}\rangle=\langle \chi_{\mathrm{far},h}f_2^{\mathrm{cut}},\chi_{\mathrm{far},h}\phi_2^{\mathrm{cut}}\rangle\\
+\langle R_h,\chi_{\mathrm{far},h}\phi_{2}^{\mathrm{cut}}\rangle+\langle[\mathcal{N}_{h,\alpha},\chi_{\mathrm{far},h}]\phi_{2}^{\mathrm{cut}},\chi_{\mathrm{far},h}\phi_{2}^{\mathrm{cut}}\rangle\,.
\end{multline*}
The only new term satisfies
\[\langle \chi_{\mathrm{far},h}f_2^{\mathrm{cut}},\chi_{\mathrm{far},h}\phi_2^{\mathrm{cut}}\rangle=\mathscr{O}(h^\infty)\|\phi_2^{\mathrm{cut}}\|^2\,,\]
since we can apply Proposition \ref{prop.locs} to the eigenfunction $f$ and use \eqref{eq.controlf}. Following the same lines as in the proof of Proposition \ref{prop.locs}, we get \eqref{eq.locs2bis}.

Let  us explain \eqref{eq.hDsoptbis}. Let us consider the case $|(\alpha_1,\alpha_2,\alpha_3)|=1$. When $\alpha_3=1$ or $\alpha_2=1$, the estimate comes from Proposition \ref{prop.controlf}. Then, we recall \eqref{eq.f-alpha=1}  and \eqref{eq.controlf} and we get \eqref{eq.hDsoptbis} with $\alpha=(1,0,0)$. For $|(\alpha_1,\alpha_2,\alpha_3)|\geq 2$, the result follows by induction.
	\end{proof}

\subsubsection{Proof of Proposition \ref{prop.quasimode2}}
We have
\[(\mathcal{M}_{h,\alpha}-\lambda)\varphi=f\,,\quad\mbox{ with }\quad (\mathcal{M}_{h,\alpha}-\lambda)f=0\,.\]
Then,  from Proposition \ref{prop.controlf} and \eqref{eq.phi2f2}, we get that
\[(\mathscr{M}_{h,\alpha}-\lambda)f_2^{\mathrm{cut}}=\mathscr{O}(h^\infty)\,,\quad (\mathscr{M}_{h,\alpha}-\lambda)\phi_2^{\mathrm{cut}}=f_2^{\mathrm{cut}}+\mathscr{O}(h^\infty)\,. \]
Thus,
\[(\mathscr{M}_{h,\alpha}-\lambda)^2\phi^{\mathrm{cut}}_2=\mathscr{O}(h^\infty)\,.\]
Recalling \eqref{eq.MNP} we get
	\[(\mathscr{N}_{h,\alpha}-\lambda)^2\phi_2^{\mathrm{cut}}
=-P_h(\mathscr{N}_{h,\alpha}-\lambda)\phi_2^{\mathrm{cut}}
-(\mathscr{N}_{h,\alpha}-\lambda)P_h\phi_2^{\mathrm{cut}}-P^2_h\phi_2^{\mathrm{cut}}
+\mathscr{O}(h^\infty)\,.\]
By means of Proposition \ref{prop.locoptimal}, we deduce Proposition \ref{prop.quasimode2}. Let us explain this. Among the terms on the right-hand side (coming from the definition of $P_h$), we have  to estimate the following
\begin{equation*}
	\begin{split}
\|r_{h,4}s^3(\mathscr{N}_{h,\alpha}-\lambda)\phi_2^{\mathrm{cut}}\|\leq& C|\lambda|\|s^3\phi_2^{\mathrm{cut}}\|+C\|s^3(hD_u)^2\phi_2^{\mathrm{cut}}\|
+C\|s^3(hD_s)^2\phi_2^{\mathrm{cut}}\|\\
&+C\|hs^3u hD_s\phi_2^{\mathrm{cut}}\|+C\|s^3u\phi_2^{\mathrm{cut}}\|+C\|s^5\phi_2^{\mathrm{cut}}\|\\
\leq & Ch^{\frac{13}{6}}\,.
\end{split}
\end{equation*}
All the other terms can be analyzed in the same way. It appears that the order of magnitude $h^{\frac{13}{6}}$ is the biggest one among all powers of $h$ appearing in the remainders.  This completes the proof of Proposition \ref{prop.quasimode2}.

\subsection{Proof of Propositons \ref{prop.locspec} \& \ref{prop.finale}}\label{sec.4.3}

\subsubsection{Proof of Proposition \ref{prop.locspec}}\label{sec.Riesz}
 Let us first consider $\varphi =(\varphi_1,\phi_2)$ an eigenfunction of $\mathscr{M}_{h,\alpha}$.
From Proposition \ref{prop.quasimode}, we have
\[\|(\mathscr{N}_{h,\alpha}-\lambda)\phi_2^{\mathrm{cut}}\|\leq \tilde Ch^{\frac32-\eta}\|\phi_2^{\mathrm{cut}}\|\,.\]
Then, we use the resolvent estimate of Proposition \ref{prop.resolventmodel}. We write $\lambda=z_1e^{\frac{2i\alpha}{3}}h^{\frac23}+\zeta h$, with $\zeta\in D(0,R)$. If $\zeta$ does not belong to the spectrum of our complex harmonic oscillator, then, we have
\[1\leq \tilde C h^{\frac32-\eta}\left(Ch^{-\frac23}+\frac{C}{\mathrm{dist}(\mathrm{sp}(h^2D_s^2+e^{i\alpha}\frac{k_0 s^2}{2}),\zeta h)}\right)\,.\]
We deduce that
\[\mathrm{dist}\left(\zeta,\{(2n-1)e^{\frac{i\alpha}{2}}\sqrt{\frac{k_0}{2}}\,,1\leq n\leq N\}\right)\leq Ch^{\frac12-\eta}\,,\]
which implies \eqref{eq.locspec}.

Let us now discuss the rank of the Riesz projector $\Pi_{n,h}$. From the estimate \eqref{eq.locspec}, we can draw the circle $\mathscr{C}_{n,h}$ with center $\mu_n(h,\alpha)$ and radius $h^{\frac32-3\eta}$ in the resolvent set of $\mathscr{M}_{h,\alpha}$.  Let us assume that the rank of the projector is at least two. There are two possibilities. Either there are two distinct (possibly not simple) eigenvalues (which coincide with $\mu_n(h,\alpha)$ modulo $\mathscr{O}(h^{\frac32-\eta})$), or there is an eigenvalue with algebraic multiplicity at least two. The strategy is to evaluate the Riesz projector on the corresponding (possibly generalized) eigenfunctions
\[
\hat\Pi_{n,h} = \frac{1}{2i\pi}\int_{\mathscr{C}_{n,h}}(z-\mathscr{N}_{h,\alpha})^{-1}
\mathrm{d}z\,,
\]
whose rank is one by Proposition \ref{prop.resolventmodel}.

Consider the first case and $\psi$ and $\tilde\psi$ corresponding normalized eigenfunctions. Let us denote $F_h=\mathrm{span}(\psi,\tilde\psi)$, which is of dimension two. Then, the map $Q_h : F_h\ni f\mapsto\chi_h f_2$ is injective. Indeed, from the Agmon estimates satisfied by the eigenfunctions, we see that
\[\|\chi_h f_2\|=\|f\|+\mathscr{O}(h^\infty)\|f\|\,,\]
and, in particular, for $h$ small enough,
\[\|f\|\leq  2\|Q_h f\|\,.\]
Moreover, we have, for all $f\in F_h$,
\[\|(\mathscr{N}_{\alpha,h}-\lambda) Q_h   f\|\leq \tilde Ch^{\frac32-\eta}\| Q_h  f\|\,,\quad \lambda=\mu_n(h,\alpha)\,.\]
We notice that
\[\hat\Pi_{n,h} Q_h  f= Q_h   f+\frac{1}{2i\pi}\int_{\mathscr{C}_{n,h}}(z-\lambda)^{-1}(z-\mathscr{N}_{h,\alpha})^{-1}
\left(\mathscr{N}_{h,\alpha}-\lambda\right) Q_h   f\mathrm{d}z\,.\]
Thus,
\[\|\hat\Pi_{n,h} Q_h  f- Q_h   f\|\leq Ch^{\eta}\| Q_h  f\|\leq\frac12\| Q_h  f\|\,.\]
This shows that $\mathrm{rank}\,\hat\Pi_{n,h}=2$, which is a contradiction.

Let us now consider the second case of an eigenvalue with algebraic multiplicity at least two. This implies the existence of   $\varphi = (\varphi_1,\phi_2) \Bk\in\ker(\mathscr{M}_{h,\alpha}-\lambda)^2$ such that $\varphi\notin\ker(\mathscr{M}_{h,\alpha}-\lambda)$.

Then, we write
\[\hat\Pi_{n,h}\phi_2^{\mathrm{cut}}=\phi_2^{\mathrm{cut}}
+\frac{1}{2i\pi}\int_{\mathscr{C}_{n,h}}(z-\lambda)^{-2}(z-\mathscr{N}_{h,\alpha})^{-1}
(\mathscr{N}_{h,\alpha}-\lambda)^2\phi_2^{\mathrm{cut}}\mathrm{d}z\,.\]
Combining Propositions \ref{prop.quasimode2} and \ref{prop.resolventmodel}, we get
\[\|\hat\Pi_{n,h}\phi_2^{\mathrm{cut}}-\phi_2^{\mathrm{cut}}\|=o(1)\|\phi_2^{\mathrm{cut}}\|\,,\]
 where we used the fact that $13/6>2$. We conclude that the range of $\hat\Pi_{n,h}$ has dimension at least two. This is a contradiction with Proposition \ref{prop.resolventmodel}. \Bk

\subsubsection{Proof of Proposition \ref{prop.finale}}\label{sec.quaRiesz}
 Considering the result of Proposition \ref{prop.locspec}, it is sufficient to show that for each fixed $n \in \{1, \cdots N\}$, the Riesz projector $\Pi_{n,h}$ is not zero.
For this consider the function
\[\psi_h(x)=(0,\chi(s,u)\Psi_{1,n,h}(s,u))\,,\]
where $\Psi_{1,n,h}$ is defined in \eqref{eq.Psimn} and $\chi$ is a smooth function with compact support equal to $1$ near $0$ and equal to $0$ outside a small neighborhood of $(0,0)$.

Then, we consider
\[\Pi_{n,h}\psi_h=\frac{1}{2i\pi}\int_{\mathscr{C}_{n,h}}(z-\mathscr{M}_{h,\alpha})^{-1}\psi_h\mathrm{d}z=\frac{1}{2i\pi}\int_{\widetilde{\mathscr{C}}_{n,h}}(z-\mathscr{M}_{h,\alpha})^{-1}\psi_h\mathrm{d}z\,,
\]
where $\widetilde{\mathscr{C}}_{n,h}$ is a circle with the same center as $\mathscr{C}_{n,h}$, but with radius of order $\epsilon h$ for $\epsilon$ small enough.
Given $z\in\widetilde{\mathscr{C}}_{n,h}$, we consider $\varphi_{h,z} = (\varphi_{h,z,1}, \phi_{h,z,2})$ the unique solution of
\[(z-\mathscr{M}_{h,\alpha})\varphi_{h,z}=\psi_h\,.\]
Then, $\varphi_{h,z}$ satisfies the Agmon estimates with a right-hand side \eqref{eq.Agmonfa} and \eqref{eq.Agmonfb}. In particular, we have, in $H^1$-norm,
\begin{equation}\label{eq.phihz1}
\varphi_{h,z,1}=\mathscr{O}(h^\infty)(\|\varphi_{h,z}\|+\|\psi_h\|)
\end{equation}
and, with similar notations as in \eqref{eq.cut},
\begin{equation}\label{eq.phihz2}
	\phi_{h,z,2}=\phi_{h,z,2}^{\mathrm{cut}}+\mathscr{O}(h^\infty)(\|\varphi_{h,z}\|+\|\psi_h\|)\,.
\end{equation}
One needs to estimate $\|\varphi_{h,z}\|$. To do so, let us consider $\phi_{h,z,2}^{\mathrm{cut}}$, which satisfies
\[(\mathscr{M}_{h,\alpha}-z)\phi^{\mathrm{cut}}_{h,z,2}=-\psi^{\mathrm{cut}}_{h,2}+\mathscr{O}(h^\infty)(\|\varphi_{h,z}\|+\|\psi_h\|)\,.\]
As in (the beginning of) the proof of Proposition \ref{prop.controlf}, we get the following.
\begin{lemma}\label{lem.fina}
	We have
\[\|hD_u\phi_{h,z,2}^{\mathrm{cut}}\|^2\leq Ch^{\frac23}\|\phi_{h,z,2}^{\mathrm{cut}}\|^2+\|\psi_h\|\|\phi^{\mathrm{cut}}_{h,z,2}\|+\mathscr{O}(h^\infty)(\|\varphi_{h,z}\|^2+\|\psi_h\|^2)\,,\]
and
\[\|hD_s\phi_{h,z,2}^{\mathrm{cut}}\|^2+\|s\phi_{h,z,2}^{\mathrm{cut}}\|^2\leq Ch\|\phi_{h,z,2}^{\mathrm{cut}}\|^2+C\|\psi_h\|\|\phi^{\mathrm{cut}}_{h,z,2}\|+\mathscr{O}(h^\infty)(\|\varphi_{h,z}\|^2+\|\psi_h\|^2)\,.\]
\end{lemma}
 Similarly, we get the control of the second order derivative with respect to $s$.
\begin{lemma}\label{lem.finb}
We have
\[\|(hD_s)^2\phi^{\mathrm{cut}}_{h,z,2}\|\leq Ch\|\phi_{h,z,2}^{\mathrm{cut}}\|+Ch^{\frac12}\|\psi_{h,2}^{\mathrm{cut}}\|^{\frac12}\|\phi^{\mathrm{cut}}_{h,z,2}\|^{\frac12}+\mathscr{O}(h^\infty)(\|\varphi_{h,z}\|+\|\psi_h\|)\,.\]	
\end{lemma}
\begin{proof}
Adapting \eqref{eq.f-alpha=2b} with our notations gives
\begin{multline*}
	\|(hD_s)^2\phi_{h,z,2}^{\mathrm{cut}}\|^2
	\leq Ch\|hD_s\phi_{h,z,2}^{\mathrm{cut}}\|^2	+Ch\|\psi_{h,2}^{\mathrm{cut}}\|\|\phi^{\mathrm{cut}}_{h,z,2}\|\\
+Ch\|(hD_s)^2\phi_{h,z,2}^{\mathrm{cut}}\|\|hD_s\phi_{h,z,2}^{\mathrm {cut}}\|	+\mathscr{O}(h^\infty)\|(hD_s)^2\phi_{h,z,2}^{\mathrm{cut}}\|(\|\psi_h\|+\|\varphi_{h,z}\|)\,.
\end{multline*}
With the Young inequality, this gives
\begin{multline*}
	\|(hD_s)^2\phi_{h,z,2}^{\mathrm{cut}}\|^2
	\leq Ch^2\|\phi_{h,z,2}^{\mathrm{cut}}\|^2	+Ch\|\psi_{h,2}^{\mathrm{cut}}\|\|\phi^{\mathrm{cut}}_{h,z,2}\|
	+\mathscr{O}(h^\infty)(\|\psi_h\|^2+\|\varphi_{h,z}\|^2)\,.
\end{multline*}
The proof is complete.
	\end{proof}
We first write
\begin{equation*}
	(\mathscr{N}_{h,\alpha}-z)\phi^{\mathrm{cut}}_{h,z,2}=-\psi^{\mathrm{cut}}_{h,2}+R_{h,z}\,,
\end{equation*}
with
\[\|R_{h,z}\|\leq   Ch^{\frac23-\eta}\|(hD_s)^2\phi^{\mathrm{cut}}_{h,z,2}\|+ Ch\|(hD_s)\phi_{h,z,2}^{\mathrm{cut}}\|+C\|s^3\phi_{h,z,2}^{\mathrm{cut}}\|
+\mathscr{O}(h^\infty)(\|\varphi_{h,z}\|+\|\psi_h\|)\,.\]
From Lemmas \ref{lem.fina} and \ref{lem.finb}, we get
\[\|R_{h,z}\|\leq Ch^{\frac23-2\eta+\frac12}\|\phi_{h,z,2}^{\mathrm{cut}}\|+Ch^{\frac23-2\eta}\|\psi_h\|^{\frac12}\|\phi_{h,z,2}^{\mathrm{cut}}\|^{\frac12}
+\mathscr{O}(h^\infty)(\|\varphi_{h,z}\|+\|\psi_h\|)\,.\]
By using Proposition \ref{prop.resolventmodel} and the fact that $z\in\widetilde{\mathscr{C}}_{n,h}$, we infer that
\[\|\phi_{h,z,2}^{\mathrm{cut}}\|\leq Ch^{-1}\|\psi_h\|+\mathscr{O}(h^\infty)(\|\varphi_{h,z}\|+\|\psi_h\|)\,.\]
With \eqref{eq.phihz1} and \eqref{eq.phihz2}, this gives
\[\|\varphi_{h,z}\|\leq Ch^{-1}\|\psi_h\|\,.\]
Then, we also deduce that
\[\left\|\int_{\widetilde{\mathscr{C}}_{n,h}}(\mathscr{N}_{h,\alpha}-z)^{-1}R_{h,z}\mathrm{d}z\right\|=o(1)\|\psi_h\|\,.\]
This shows that
\[\Pi_{n,h}\psi_h=\frac{1}{2i\pi}\int_{\mathscr{C}_{n,h}}(z-\mathscr{N}_{h,\alpha})^{-1}\psi^{\mathrm{cut}}_{h,2}\mathrm{d}z+o(1)\|\psi_h\|\,.
\]
Recalling the resolvent formula
\[(z-\mathscr{N}_{h,\alpha})^{-1}-(z-\mu_n(h,\alpha))^{-1}=(z-\mathscr{N}_{h,\alpha})^{-1}(z-\mu_n(h,\alpha))^{-1}(\mathscr{N}_{h,\alpha}-\mu_n(h,\alpha))\,,\]
and that
\[(\mathscr{N}_{h,\alpha}-\mu_n(h,\alpha)) \psi_{h,2}^{\mathscr{cut}}=\mathscr{O}(h^\infty)\,,\]
we get
\[\Pi_{n,h}\psi_h=\psi_h+o(1)\|\psi_h\|\,.\]
Therefore, $\Pi_{n,h}$ is not zero for $h$ small enough. Recalling the discussion
at the beginning of this section, this completes the proof of Proposition \ref{prop.finale}.

\section*{Acknowledgments}
This work was conducted within the France 2030 framework programme, the Centre Henri Lebesgue  ANR-11-LABX-0020-01. The authors thank the F\'ed\'eration  de recherche
Math\'ematiques des Pays de Loire and its regional project Ambition Lebesgue Loire, which funded the stay of D. K. at the Laboratoire Angevin de Recherche Math\'ematique in May 2022, when this work started.
D.K. was also partially supported by the EXPRO grant No. 20-17749X
of the Czech Science Foundation. N.R. is grateful to Anne-Sophie Bonnet - Ben Dhia and Marc Lenoir for stimulating discussions at the Oberwolfach Research Institute for Mathematics in September 2022. He also wishes to thank Martin Averseng for enlightening discussions about the Caccioppoli estimates and the presentation in Section \ref{sec.conse}. The authors also thank Laura Baldelli for communicating them references about Caccioppoli estimates. F. H. and N.R. are also grateful to Bernard Helffer for useful discussions.

\bibliographystyle{abbrv}
\bibliography{biblio}

\end{document}